\documentclass[a4paper,10pt]{amsart}
\usepackage[latin1]{inputenc}
\usepackage{graphicx}
\usepackage{caption}
\usepackage{subcaption}
\usepackage{mathtools}

\mathtoolsset{showonlyrefs} 
\usepackage[T1]{fontenc}
\usepackage{bbm}
\usepackage{amsmath}
\usepackage{amsthm}
\usepackage{amssymb}
\usepackage{color}
\usepackage{latexsym}
\usepackage{ulem}
\usepackage[all]{xy}
\usepackage{calc}
\usepackage{soul} 
\usepackage{multicol}
\newtheorem{thm}{Theorem}[section]
\newtheorem{prop}{Proposition}[section]

\newtheorem{rmq}{Remark}[section]

\newcommand{\R}{\mathbb{R}}
\numberwithin{equation}{section}

\providecommand{\R}{\mathbb{R}}

\newcommand{\N}{\mathbb{N}}
\newcommand{\Z}{\mathbb{Z}}

\newcommand{\stkout}[1]{\ifmmode\text{\sout{\ensuremath{#1}}}\else\sout{#1}\fi}
\newcounter{exercice}

\begin{document}
\title[Traveling wave in the Helmholtz regime]{Existence of traveling waves for a fourth order Schr\" odinger equation with mixed dispersion in the Helmholtz regime}

\author[ Casteras, and F\" oldes]{ Jean-Baptiste Casteras \and Juraj F\" oldes}

\address{ Jean-Baptiste Casteras
\newline \indent  CMAFCIO, Faculdade de Ci\^ encias da Universidade de Lisboa,
\newline  \indent Edificio C6, Piso 1, Campo Grande 1749-016 Lisboa, Portugal.}
\email{jeanbaptiste.casteras@gmail.com}

\address{Juraj F\" oldes
\newline \indent Department of Mathematics, University of Virginia \indent 
\newline \indent  322 Kerchof Hall, 141 Cabell Drive,  Charlottesville, VA 22904-4137,\indent }
\email{foldes@virginia.edu}

\thanks{J. B. Casteras is supported by FCT - Funda\c c\~ao para a Ci\^encia e a Tecnologia, under the project: UIDB/04561/2020;
J. F\" oldes is partly supported by the National Science Foundation under the grant NSF-DMS-1816408.}

\begin{abstract}
In this paper, we study the existence of traveling waves for a fourth order Schr\" odinger equations with mixed dispersion, that is, solutions to
$$\Delta^2 u +\beta \Delta u +i V \nabla u +\alpha u =|u|^{p-2} u,\ in\ \R^N ,\ N\geq 2.$$
We consider this equation in the Helmholtz regime, when the Fourier symbol $P$ of our operator is strictly negative at some point. Under suitable assumptions, 
we prove the existence of solution using the dual method of Evequoz and Weth provided that $p\in (p_1 , 2N/(N-4)_+)$. The real number $p_1$ depends on the number of principal curvature of $M$ 
staying bounded away from $0$, where $M$ is the hypersurface defined by the roots of $P$. We also obtained estimates on the Green function of our operator and a $L^p - L^q$ resolvent estimate 
which can be of independent interest and can be applied to other operators. 
\end{abstract}

\maketitle

\section{Introduction}
In this paper, we  construct non-trivial solution to a fourth order nonlinear equation
\begin{equation}
\label{maineq}
\Delta^2 u +\beta \Delta u +i V \nabla u +\alpha u =|u|^{p-2} u,\qquad \textrm{in}\quad \R^N ,\ N\geq 2,
\end{equation}
where $\alpha ,\beta \in \R$, $V\in \R^N$ and $p>2$. The equation \eqref{maineq} characterizes the profile of traveling wave solutions  $\varphi (t,x)=e^{i\alpha t}u(x-vt)$ of the fourth order nonlinear Schr\" odinger equation with mixed dispersions
\begin{equation}
\label{4NLSdis}
i\partial_t \varphi - \Delta^2 \varphi -\beta \Delta \varphi + |\varphi|^{p-2} \varphi =0,\ \varphi (0,x)=\varphi_0 (x),\ (t,x)\in \R\times \R^N.
\end{equation}
The fourth order term in equation \eqref{4NLSdis} has been introduced by Karpman and Shagalov \cite{MR1779828} and it allows to regularize and stabilize solutions to the classical Schr\" odinger equation as observed through numerical simulations by Fibich, Ilan, and Papanicolaou \cite{MR1898529}. Let us briefly mention selected results for \eqref{4NLSdis}. 
Well-posedness has been established by Pausader \cite{MR2353631} (using the dispersive estimates of \cite{MR1745182}) as well as some scattering results (we also refer to \cite{miao}). Recently, Boulenger and Lenzmann obtained (in)finite time blowing-up results in \cite{BL}.

 In the last few years, solitary waves solutions to \eqref{4NLSdis}, that is, solutions of the form $\varphi (t,x)=e^{i\alpha t}u(x)$, have been quite intensively studied. The profile of such waves 
 satisfies \eqref{maineq} with $V \equiv 0$ and we refer to \cite{MR3855391,MR3976588,MR4001029,MR3494890,FJMM} for several results concerning the existence of ground states and normalized solutions as well as some of their qualitative properties. Notice that unlike the classical second order Schr\" odinger equation, it does not seem possible to  transform standing  to traveling waves. 
Thus,  \eqref{maineq} with $V\neq 0$ should be studied separately. 
 Several results have been obtained in this direction by the first author in \cite{C}. Specifically, the existence of normalized solutions, that is, solutions with fixed mass and the existence of solutions with fixed mass and momentum was established under suitable assumptions. In such ansatz, the parameter $\alpha$ (or the parameters $\alpha$ and $V$) is not fixed and appears as a Lagrange multiplier. The existence of ground state solutions as well as some of their qualitative properties were also investigated. We call $u\in H^2 (\R^N))$ a ground state solution if $u$ minimizes the energy 
$$E(f)= \dfrac{1}{2} \int_{\R^N} (|\Delta f|^2 - \beta |\nabla f|^2 +i  \bar{f} V\nabla f+ \alpha |f|^2 )dx - \dfrac{1}{p}\int_{\R^N} |f|^p dx,$$   
among all function $f\in H^2 (\R^N)\backslash \{0\}$. We also refer to \cite{Him} (see also \cite{BLSS}) where similar questions were studied  for a larger class of operators of the form $P_v(D)=P(D)+iV \nabla$,
where $P(D)$ is a self-adjoint and constant coefficient pseudo-differential operator defined by $(\widehat{P(D) u})(\xi)=p(\xi) \hat{u} (\xi)$, $\hat{u}$ being the Fourier transform $u$. 
The existence of ground state solutions was obtained provided that $2<p < 2N/(N-2s)_+$  and $P(D)$ has a real-valued and continuous symbol $p:\R^N \rightarrow \R$ satisfying
\begin{equation}
\label{condPD}
A |\xi|^{2s} +c \leq p(\xi) \leq B |\xi|^{2s} \qquad \textrm{ for all } \xi \in \R^N ,
\end{equation}
for some constants $s>1/2$, $A>0$, $B>0$ and $c\in \R$. 
Condition \eqref{condPD} in particular implies that the energy $E$ is bounded from below and coercive.

  The main goal of the present manuscript is to obtain solutions to \eqref{maineq} when \eqref{condPD} is not satisfied. The main difficulty
stems from the fact that $0$ is contained in the essential spectrum of the operator $L= \Delta^2 +\beta \Delta +iV\nabla +\alpha$. To solve this obstacle, we use the dual variational method due to Evequoz and Weth \cite{EW}. 
 More precisely, we look for solution to
\begin{equation}
\label{eqdualintro}
\mathfrak{R}(v)= |v|^{p^\prime -2}v,\ in\ \R^N ,
\end{equation}
where $v=|u|^{p-2}u$, $p^\prime= p/(p-1)$ and $\mathfrak{R}=(\Delta^2 +\beta \Delta +iv\nabla +\alpha)^{-1}$ is a resolvent-type operator constructed by a limiting absorption principle. 
The advantage of \eqref{eqdualintro} is that it has
 variational structure in $L^{p^\prime}(\R^N)$ and one can use  a mountain pass theorem to construct a non-trivial solution to \eqref{maineq}. Our main result is the following:
 
\begin{thm}
\label{mainthm}
Let $P$ be the Fourier symbol of the differential operator in \eqref{maineq} without the drift term, that is, $P(x)=x^4 -\beta x^2 +\alpha$ and let
$$M:=\{\xi \in \R^N : P(|\xi|)-\xi V=0 \}.$$
Assume that :
\begin{itemize}
\item (A1) there exists $\tilde{\xi}\in \R^N$ such that $ P(|\tilde{\xi}|)-\tilde{\xi} V <0$;
\item (A2) $P^\prime (|\xi|)\frac{\xi}{|\xi|} -V \neq 0$, for all $\xi \in M$. 
\item (A3) If $P^\prime (|\xi|)=0$, then $P(|\xi|) - |V||\xi| \neq 0$.
\end{itemize}
Let $p_1<p<2N/(N-4)_+ $ with
\begin{equation}
p_1 = 
\begin{cases}
\frac{2(N+1)}{N-1} & \textrm{if }8 \beta |V|^{2} \leq  (\beta^2 -4\alpha)^2 \textrm{ and }\beta^2 - 4\alpha \neq 3 |V|^{4/3} \,,\\
\frac{2N}{N-2} & \textrm{if } 8 \beta |V|^{2} >  (\beta^2 -4\alpha)^2\textrm{ or } \beta^2 - 4\alpha = 3 |V|^{4/3} \,.
\end{cases}
\end{equation}
Then, there exists a nontrivial solution $u\in W^{4,q}(\R^N) \cap C^{3,\alpha}(\R^N)\cap L^p (\R^N)$, for all $q\in [p,\infty)$, $\alpha \in (0,1)$ to
$$\Delta^2 u +\beta \Delta u +i V \nabla u +\alpha u =|u|^{p-2} u,\ in\ \R^N . $$
\end{thm}

We refer to Remark \ref{rmq} for necessary and sufficient conditions on $\alpha,\beta$ and $V$ for (A1), (A2), and (A3) to be satisfied.

Let us make a few comments on this theorem. Assumption (A1)  guarantees that we are in the Helmholtz case, that is, \eqref{condPD} is not satisfied, thus if removed, our main result is 
already covered in the literature. 
Assumption (A2) implies that $M$ is a smooth manifold, and it is a mandatory condition for our method to work. 
Otherwise,  even the definition of the resolvent type operator would be problematic if $P(x)$ had a double root. We remark that if (A2) is not satisfied, then either $M$ is a single point, or
a union of two smooth touching manifolds. 

The assumption (A3) allows us to control the number of non-vanishing curvatures of $M$ located on the axis of rotation (see the case $h(t)=0$ in the proof of Proposition \ref{propM} below). In particular, if (A3) is not satisfied, then all principal curvatures
vanish at the axis of rotation, which leads to insufficient estimates for our purposes. 

Next, the value of $p_1$ is  linked to the number $k$ of principal curvatures of $M$ staying bounded from $0$. An interesting feature of our problem is that, depending on the coefficients of $P$, 
$k$ is either $N-1$ (similar to the "classical" situation where $M$ is a $(N-1)$-sphere) or $N-2$. Let us remark that our method can be applied to more general self-adjoint, constant coefficient operators. 
More precisely, assume that the Fourier symbol of the operator is given by $\tilde{P}$ and $\tilde{M}= \{x\in \R^N : \tilde{P}=0\}$ is a not empty regular compact hypersurface with $k$ principal curvature bounded away from $0$. 
Then,  \eqref{maineq} with the changed operator (corresponding to $\tilde{P}$) admits a solution provided $2(k+2)/k<p<2N/(N-s)_+$, where $s>0$ is the order of the operator. 
We remark that we expect our solution to decay as $|x|^{-k/2}$ at infinity, in analogy with the situation for the nonlinear Helmholtz equation (see \cite{MMP}).  

 Let us make some comments on the proof of Theorem \ref{mainthm}. The main challenge is to construct the resolvent type operator $\mathfrak R$ and to analyze its mapping properties. To do so, we follow the approach of Gutierrez \cite{Gu}, which relies on two main ingredients. The first one is the decay of the Green function of the operator $\Delta^2 +\beta \Delta +iV\nabla +\alpha$, which  is connected with 
 $k$, the number of principal curvature staying bounded away from 0. This can be already seen from the classical Fourier restriction result of Littman \cite{MR155146} (see Theorem \ref{restthm}). To prove 
the decay,  we follow  the framework of  Mandel \cite{MR3949725}, that was developed for a periodic differential equation of Helmholtz type. 

The second main ingredient, in the analysis of the resolvent operator,  is the Stein-Tomas Theorem (see Theorem \ref{STthm}  below proved in \cite{MR2831841})
 The main technical problem is  the proof $L^{2(k+2)/(k+4)} (\R^N) \rightarrow L^2 (\R^N)$ estimate (see \eqref{RT2} below) for which  we used a radial dyadic decomposition. 
 After applying the Stein-Tomas Theorem, we need an estimate
\begin{equation}\label{iig}
\int_{-3c_1/4}^{3c_1 /4}  \max_{\xi \in M_\tau}(|\hat{G}_2^j (\xi )|^2) d\tau \leq C2^j\,,
\end{equation} 
where $G_2^j$ is roughly the Green function 
 cut-off on to the annulus $B_{2^{j+1}} \backslash B_{2^{j-1}}$. This approach was employed in
  \cite{Gu} and \cite{MR3977892}, where the dyadic decomposition has the same symmetry as $M=\mathbb{S}^{N-1}$, which led to the left hand side of \eqref{iig} without $\max$.  
  Then, an application of Plancherel Theorem led to the desired result. However, our situation is more complicated and the Plancherel theorem, or Bernstein type inequality yields only insufficient bound by 
  $C2^{Nj}$. 
   To prove \eqref{iig}, we work directly in the Fourier space and we use a non-trivial cancellation properties that originate from the non-degeneracy assumption (A2). 
   We stress, that our estimate is the same as the radial case, although working in more general setting. 
   We also remark that the mapping properties of our resolvent operator can be made independent of the dimension $N$ (with slight loss of generality)
   and as such they depend only on $k$

The plan of this paper is the following: in Section $2$, we study the properties of $M$. In particular, we give conditions on the coefficients of our equation implying that the number of principal curvatures of $M$ bounded away from $0$ is either $N-1$ or $N-2$. Then, we use these properties to study the decay of the Green function of our operator. In Section $3$, we study the mapping property of the resolvent type operator $(\Delta^2 +\beta \Delta +iv\nabla +\alpha)^{-1}$. Finally, we use them to construct non-trivial solutions to \eqref{maineq} using the dual variational method.  

\noindent {\em Acknowledgement} : We would like to thank Rainer Mandel and Gennady Uraltsev for very useful and inspiring discussions.

\section{Decay estimate of the Green function}
The aim of this section is to analyse the decay of the Green function $G$ of the operator $\Delta^2 +\beta \Delta +iV\nabla +\alpha$, that is,
$$\Delta^2 G +\beta \Delta G +i V  \nabla G +\alpha G= \delta_0,\ \textrm{in }\ \R^N \,,$$
where $\delta_0$ is the Dirac measure located at the origin. 
Observe that formally, applying the Fourier transform to the previous equation, we obtain
$$G(x)=  \int_{\R^N} \dfrac{1}{P(|\xi |) - \xi V } e^{ix\xi} d\xi, $$
where $P(x)=x^4 -\beta x^2 +\alpha$. 
This integral is not well defined for general $P$; 
however, due to our assumptions on $P(|\xi |) - \xi v $, we show non-trivial cancellation properties, which allow us to prove that $G$ is well defined and decays at infinity. 
As usual for this kind of problem, we apply the "limiting absorption principle". Specifically, we define $G$ as a limit of approximating Green functions 
$G_\varepsilon$, $\varepsilon \neq 0$ associated with the operator  $\Delta^2 +\beta \Delta +i V \nabla +(\alpha -i \varepsilon)$, that is, 
$$G_\varepsilon (x)= \int_{\R^N} \dfrac{1}{P(|\xi |) - \xi V -i\varepsilon} e^{ix\xi} d\xi.$$
Notice that the denominator in the integral does not vanish for any $\varepsilon \neq 0$. 
It is well-known that the decay of this integral depends only on the set of points where $P(|\xi |)  -\xi V = 0$. More precisely, we prove that it depends on the number of non-zero principal curvatures $M$. 
Our first result gives optimal conditions on the coefficients of $P$ so that the principal curvatures of $M$ are strictly positive.  

\begin{prop}
\label{propM}
Let $P(x)=x^4 -\beta x^2 +\alpha$ and
$$M:=\{\xi \in \R^N : P(|\xi|)-\xi V = 0 \}.$$
If (A1), (A2), and (A3) hold, then the following assertions are true. 
\begin{itemize}
\item If 
\begin{equation}
\label{assumption1}
8 \beta |V|^{2} \leq  (\beta^2 -4\alpha)^2\ and\ \beta^2 - 4\alpha \neq 3 |V|^{4/3},
\end{equation}
 then $M$ is a regular, compact hypersurface with all principal curvatures bounded away from $0$. 
\item If 
\begin{equation}
\label{assumption2}
8 \beta |V|^{2} >  (\beta^2 -4\alpha)^2\ or \ \beta^2 - 4\alpha = 3 |V|^{4/3},
\end{equation}
then  exactly one of the principal curvature of $M$ vanishes on a $N-2$-dimensional set.
\end{itemize}
\end{prop}

\begin{rmq}
\label{rmq}
In this remark, we find necessary and sufficient conditions on the coefficients of our operator to satisfy assumptions (A1), (A2), and (A3). Let us begin with (A2). Assume by contradiction that there exists $\xi \in M$ such that $P^\prime (|\xi|)\frac{\xi}{|\xi|} -V = 0$. Therefore $\xi$ is a multiple of $V$ so we set $\xi = \frac{V}{|V|}q$  and obtain 
that $q$ satisfies $P'(q) - |V| = 0$, or equivalently
\begin{equation}\label{deqq}
4 q^3 - 2\beta q - |V| =0.
\end{equation}
Since $\xi = \frac{V}{|V|}q \in M$, then $|\xi|^4 - \beta |\xi|^2 + \alpha - V \xi = 0$ and in addition to \eqref{deqq} we have
\begin{equation} \label{oefq}
q^4 - \beta q^2 + \alpha - |V| q = 0 \,.
\end{equation}
Using the Euclidean algorithm for finding the greatest common divisor of these two polynomials, we obtain that they have a common root if and only if 
\begin{equation}\label{tcn}
256 \alpha^3 - 128 \alpha^2 \beta^2 + 4\beta^3 |V|^2 - 27 |V|^4 + 16 \alpha (\beta^4 - 9 \beta |V|^2) = 0 \,.
\end{equation}
Although, technically complicated, it is easy to verify if a given set of parameters satisfies \eqref{tcn}. Also, notice that \eqref{tcn} is a quadratic equation for $|V|^2$.

To satisfy assumption (A1), we need $P(|\xi|) - V\xi < 0$ for some $\xi$, and therefore we minimize the left hand side by choosing $\xi = \frac{V}{|V|} |\xi|$
\begin{equation}
|\xi|^4 - \beta |\xi|^2 + \alpha < |V| |\xi| \,.
\end{equation}
So our problem is equivalent to
\begin{equation}
|V| > \min_{x > 0} \frac{x^4 - \beta x^2 + \alpha}{x} =  \min_{x > 0} g(x) \,.
\end{equation}
If $\alpha < 0$, then $g(x) \to -\infty$ as $x \to 0$, and (A1) is trivially satisfied for any $|V|$. If $\alpha = 0$, then $g(0) = 0$, and (A1) is satisfied for any $V \neq 0$. 

Finally, if $\alpha > 0$, then $g(x) \to \infty$ for $x \to \infty$ or $x \to 0$, then the global minimum is also a local minimum. 

Evaluating $g'(x) = 0$, we find that the mininum is attained at $x$ solving
\begin{equation}
\frac{3x^4 - \beta x^2 - \alpha}{x^2} = 0 \,,
\end{equation}
that is, at since $x \geq 0$ 
\begin{equation}\label{sxpm}
 x_{\pm} = \sqrt{\dfrac{\beta \pm \sqrt{\beta^2 +12 \alpha }}{6}} \,.\ 
\end{equation}
Again, since $\alpha > 0$, then the only admissible root is $x_+$. 
Thus for $\alpha > 0$ after using $3x^4 - \beta x^2 - \alpha = 0$, we find (A1) holds if and only if  
\begin{equation}
|V| >  \frac{x_+^4 - \beta x_+^2 + \alpha}{x_+} =  \frac{2}{3} \frac{-x_+^2 + 2\alpha}{x_+} .
\end{equation}

Finally, (A3) holds if and only if
$$\alpha \neq 0 \qquad \textrm{and } \quad -\frac{\beta^2}{4} + \alpha \pm |V|\frac{\beta}{2} \neq 0. $$

\end{rmq}

\begin{proof}[Proof of Proposition \ref{propM}.]
Since $P^\prime (|\xi|)\frac{\xi}{|\xi|} -V \neq 0$, for all $\xi \in M$, and $P$ has a superlinear growth, the implicit function theorem implies that $M = \{x: P(|x|) - Vx = 0\}$ 
is a smooth compact manifold. 
To simplify notation, we choose our coordinate axes such that $V=|V|e_1$  and parametrize $M$ as a surface of revolution by 
$\xi =(\xi_1 ,\tilde{\xi})\in \R \times \R^{N-1}$. Since $(\xi_1, \tilde{\xi}) \in M$ satisfy $P((\xi_1^2 + |\tilde{\xi}|^2)^{\frac{1}{2}}) - |V| \xi_1 = 0$, by solving for $|\tilde{\xi}|$, we obtain 
$h_\pm(\xi_1) = |\tilde{\xi}|$ with $h_\pm(z)=\sqrt{(P^{-1}_\pm (|V|z))^2 -z^2 }$ and 
\begin{equation}
(P^{-1}_\pm (z))^2 = \dfrac{\beta \pm \sqrt{\beta^2 -4(\alpha - z)} }{2}.
\end{equation}
Before we proceed, let us clarify some details. Note that we choose $h$ instead of $-h$, since $|\tilde{\xi}| \geq 0$. Also, in general $P^{-1}(z)$ has four solutions, but since $P$ is even, they come 
in pairs differing by the sign. However, we only need to look $(P^{-1})^2$. Therefore it suffices to consider the two positive roots. Since we are looking at local properties (curvatures) we can 
treat two connected components (corresponding to $h_+$ and $h_-$) separately unless they touch. This happens only when a root of $P(|\xi|) - V\xi$ is degenerate, since otherwise $M$ is a smooth manifold 
(see Figure \ref{fig1}). 

\begin{figure}
\centering 
\begin{subfigure}{.35\textwidth}
\centering
\includegraphics[width=5cm]{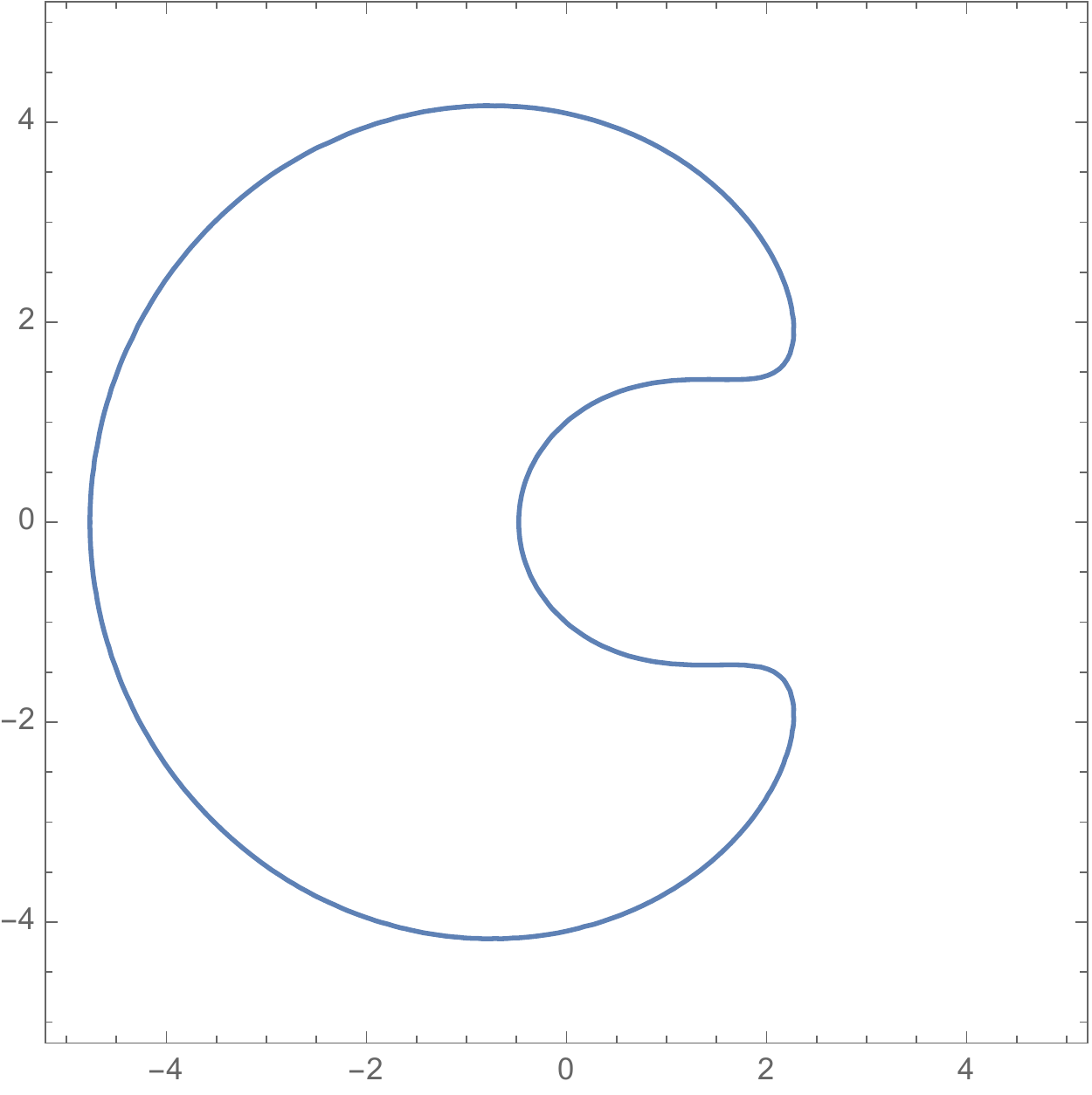}
\caption{Cross section of $M$ with one component}
\end{subfigure}
\qquad\qquad 
\begin{subfigure}{.35\textwidth}
\centering
\includegraphics[width=5cm]{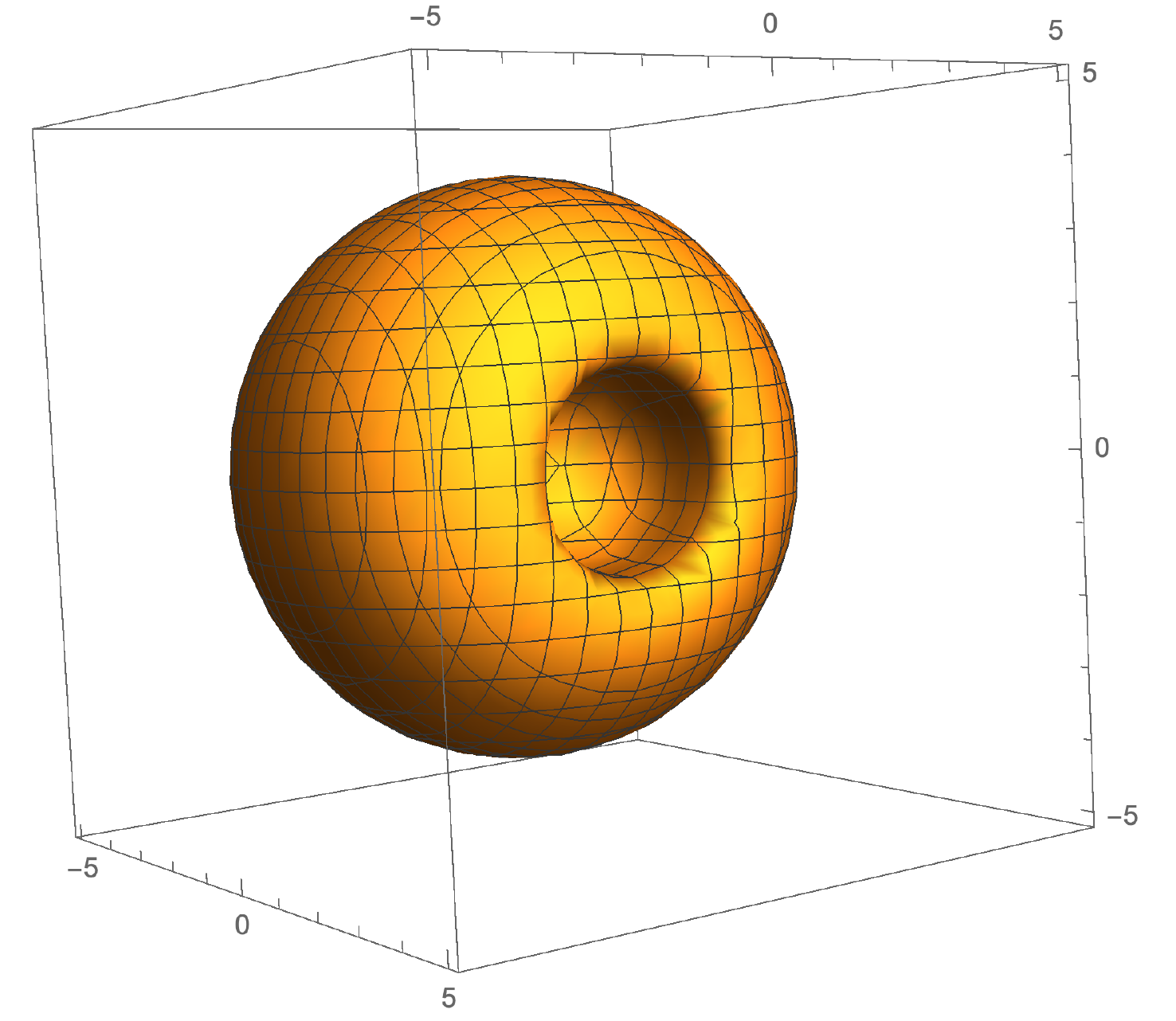}
\caption{Visualization of $M$ in 3D}
\end{subfigure}
\newline
\newline
\noindent
\centering
\begin{subfigure}{.35\textwidth}
\centering
\includegraphics[width=5cm]{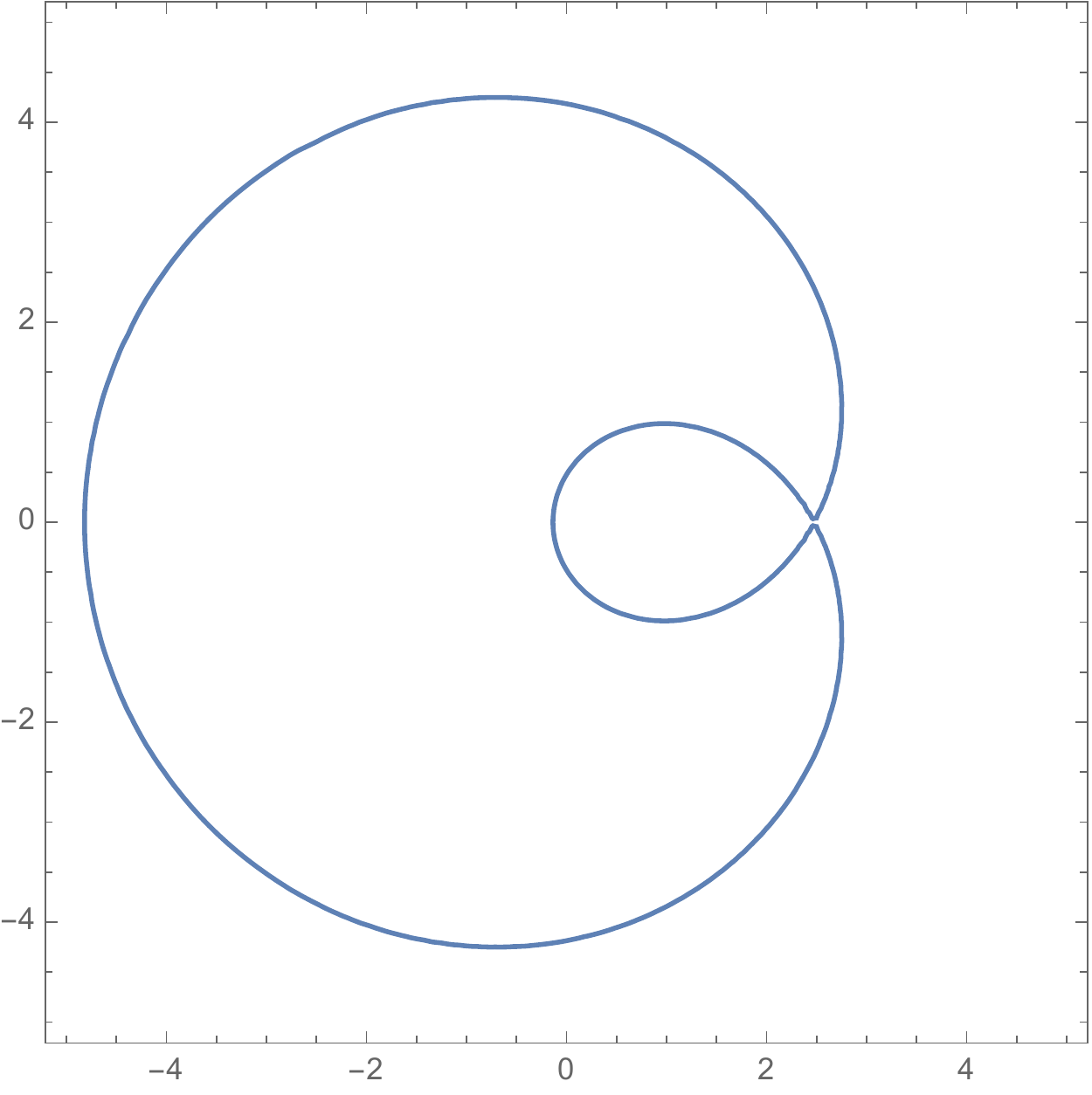}
\caption{Degenerate case}
\end{subfigure}
\qquad \qquad 
\begin{subfigure}{.35\textwidth}
\centering
\includegraphics[width=5cm]{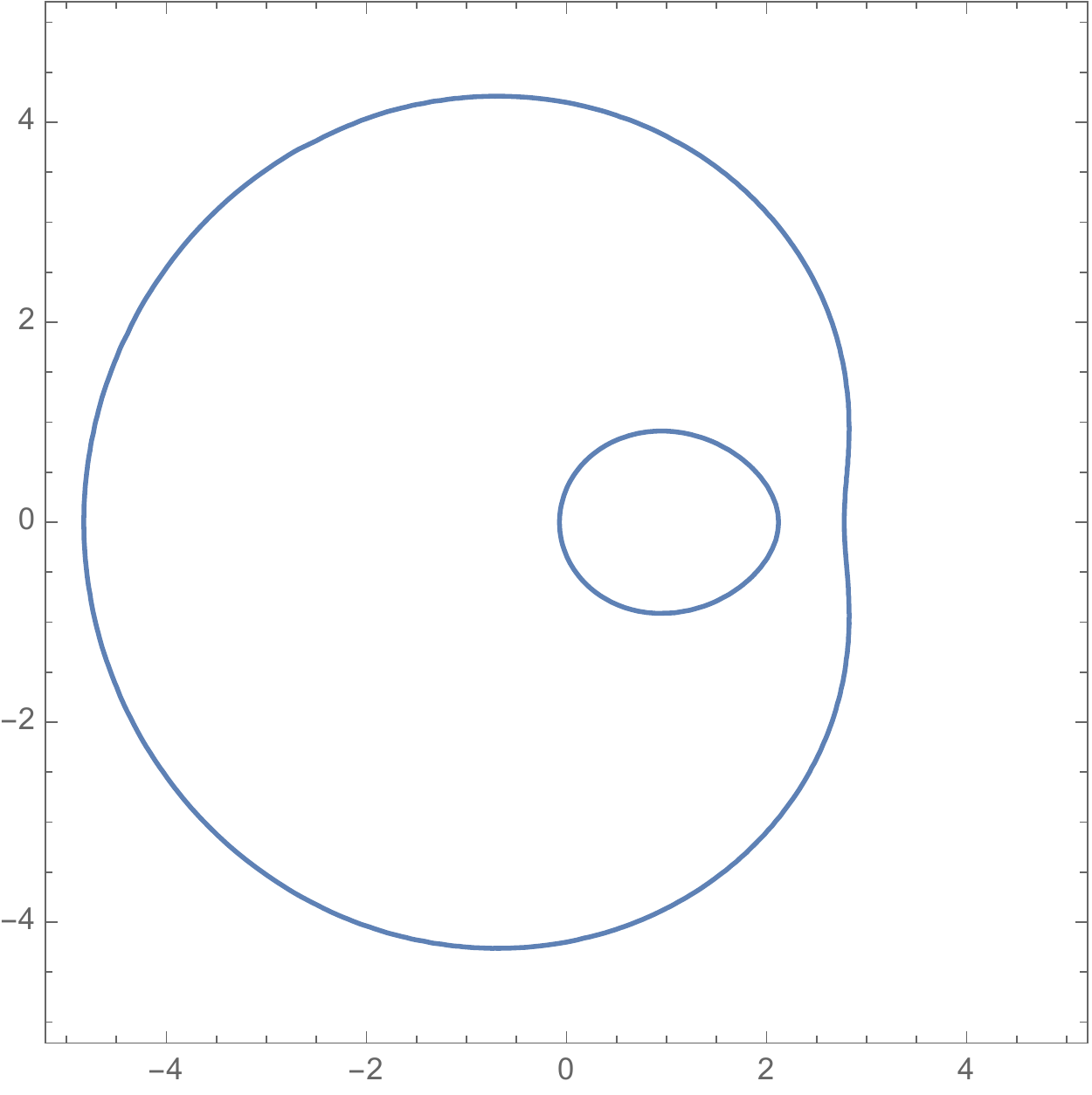}
\caption{$M$ with two components}
\end{subfigure}
\hphantom{1cm}
\caption{Figures (A), (C), (D) are cross sections of $M$ with different parameters. Figure (B) is surface of revolution
corresponding to (A)}
\label{fig1}
\end{figure}

Since $P(x) \approx x^4$ for $x$ large and $P$ is bounded from below, we obtain that $P^{-1} \approx x^{1/4}$ for $x$ large and $P^{-1}$ is defined only for $x \geq Q$ for some $Q \in \R$.
Since $P^{-1}_{\pm}(x) \geq |x|$ for any $x$ in the domain of $h_{\pm}$, we get that 
the domain of $h_\pm$ is of the form $[A_{\pm}, B_{\pm}]$ for some $A_{\pm}, B_{\pm} \in \R$. Note that one or both branches might not exist in the situations where $M$ is connected or empty. 

Thus, we can parametrize $M$ as
$$
M_\pm=\{(t,h_{\pm}(t) \omega ):\omega \in \mathbb{S}^{N-2},\ A_\pm\leq t\leq B_\pm \} \,,
$$
and then the principal curvatures $\kappa_i$, $i=1,\ldots ,N-1$, of $M$ are given (up to a sign) by 
\begin{equation}
\label{kappa}
\kappa_i (t)=\begin{cases}
\dfrac{1}{h(t) \sqrt{1+h^\prime (t)^2 }},\ &i=1,\ldots ,N-2, \\ 
\dfrac{-h^{\prime \prime} (t)}{(1+h^\prime (t)^2)^{3/2}},\ &i=N-1 \,,
\end{cases}
\end{equation}
where $h$ stands for $h_\pm$. 
Observe that $\kappa_i$, $i=1,\ldots ,N-2$, are non-zero whenever $h(t) \neq 0$. 

Before discussing $\kappa_{N - 1}$, let us treat the case where $t$ is such that $h(t) = 0$. 
By the rotational symmetry of $M$, we obtain that, at $(\xi_1, |\tilde{\xi}|) = (t, h(t)) = (t, 0)$, all principal curvatures are the same. Therefore
it suffices to prove that one of them is non-zero. Choose a smooth curve $\gamma(\xi_2) = (\xi_1(\xi_2), \xi_2, 0) \in M$, $\xi_2 \in (-\varepsilon, \varepsilon)$ with $\xi_1'(0) = 0$. Such a curve exists by the rotational symmetry. 
If $\xi_1(0) = 0$, then $\gamma(0) \in M$ yields $P(0) - 0|V| = 0$ and since $P$ is even $P'(0) = 0$, a contradiction to (A3). So, we assume $\xi_1(0) \neq 0$.  

To prove 
that $\gamma$ has nonzero curvature at $0$, and therefore $M$ has a nonzero curvature at $(\xi_1, 0)$, it suffices to prove that $\xi_1''(0) \neq 0$. However, since $\gamma \subset M$, we have 
\begin{equation}
P(|(\xi_1(\xi_2), \xi_2, 0)|) - |V| \xi_1(\xi_2) = 0 \,. 
\end{equation} 
Differentiating implicitly twice the previous equality and substituting $\xi_2 = 0$, $\xi_1'(0) = 0$, we obtain
\begin{equation}
-|V| \xi_1''(0) + P'(|\xi_1(0)|) \xi_1''(0)\frac{\xi_1(0)}{|\xi_1(0)|} +  P'(|\xi_1(0)|)\frac{1}{|\xi_1(0)|} = 0\,.
\end{equation}
Since $\xi_1(0) \neq 0$ and $V = |V|e_1$, we deduce from (A2) that 
\begin{equation}
|P'(|\xi_1(0)|)\frac{\xi_1(0)}{|\xi_1(0)|} -  |V|| = |\nabla(P(|\xi|) - V\xi)|\Big|_{\xi = (\xi_1 (0), 0)} \neq 0 \,.
\end{equation}
Finally, since by (A3), $P'(|\xi_1(0)|) = 0$ implies $P(|\xi_1(0)|) - |V|\xi_1(0) \neq 0$, that is,  $(\xi_1(0), 0) \not \in M$, we obtain that $\xi_1''(0) \neq 0$. Thus $M$ has $N - 1$ non-vanishing curvatures 
at  any point $(\xi_1, 0) \in M$.

In the rest of the proof we assume $h(t) \neq 0$. 
To treat $\kappa_{N-1}$, recall that 
$$h_\pm (z)= \sqrt{\dfrac{\beta \pm \sqrt{\beta^2 -4(\alpha - |V|z)} }{2} -z^2}.$$
First, $h_+$ is a square root of a sum of two concave functions, at least one of them strictly concave, and therefore $h''(z) < 0$ whenever defined. In particular we obtain $(\kappa_+)_{N - 1} \neq 0$.

To treat $h_-$, we set $z=|V|^{1/3}w$, $e=|V|^{-2/3} \beta$, and $k=(\beta^2 - 4\alpha)|V|^{-4/3}  $ and (noticing that $h_- \geq 0$ when it is defined)
\begin{align*}
\sqrt{f(w)} &:=|V|^{-1/3}h_- (|V|^{1/3} w)\\
&= \sqrt{\dfrac{e-\sqrt{k+4w} }{2} -w^2}.
\end{align*}
Recall $h_- (z) > 0$ and 
notice that $h_-^{\prime \prime} (z)=0$ if and only if 
\begin{equation*}
\label{nondeg}
(f^\prime (w))^2 =2 f^{\prime \prime}(w) f(w) \,,
\end{equation*}
or equivalently
\begin{equation}
\left(\frac{1}{\sqrt{k + 4w}} + 2w\right)^2 = 2 \left( \frac{2}{\sqrt{(k + 4w)^3}}  - 2\right) f(w) \,.
\end{equation}
Observe that $k + 4w = 0$ corresponds to the case $h(t) = 0$ treated separately above. 
Using the change of variables $y=(k+4w)^{1/2}$, the previous equality becomes
\begin{equation}
\label{star}
\left(\dfrac{1}{y}+ \dfrac{y^2 - k}{2}\right)^2 =4 (y^{-3} -1) f(w) = 4 (y^{-3} -1) \left(\dfrac{e-y}{2} - \left(\dfrac{y^2 -k}{4}\right)^2 \right) .
\end{equation}
Since the left hand side of \eqref{star} is non-negative and $f(w) > 0$, there are no solutions if $y > 1$. Next, for $g = \frac{e}{2} - \frac{k^2}{16}$, 
the factorization of \eqref{star}
gives
\begin{equation}
\frac{3}{y^2} + \frac{k^2}{4} - \frac{3}{4}y - \frac{3}{2}\frac{k}{y} = 4g\left(\frac{1}{y^3} - 1\right) \,.
\end{equation} 
A multiplication by $4y^2$ yields
\begin{equation}
12 + k^2y^2 - 3y^3 - 6ky = \frac{16}{y}g\left(1 - y^3\right) \,,
\end{equation}
and therefore
\begin{equation}\label{star2}
(ky -3)^2 +3(1-y^3) = \dfrac{16}{y}g (1-y^3). 
\end{equation}
Hence, if $g \leq 0$ or equivalently when $8 \beta V^{2} \leq  (\beta^2 -4\alpha)^2$, then \eqref{star2} has no solution $y < 0$ since the left hand side is positive and the right hand side is non-positive. 
Finally when $y=1$, \eqref{star2} has a solution only if $k=3$, that is, $\beta^2 - 4\alpha = 3 V^{4/3}$. Overall, we have proved that $h^{\prime \prime}\neq 0$ provided that
\begin{equation}
\label{mainCond}
8 \beta V^{2} \leq  (\beta^2 -4\alpha)^2\qquad  \textrm{and}\quad \beta^2 - 4\alpha \neq 3 V^{4/3}.
\end{equation}

Let us remark that if \eqref{mainCond} does not hold, then  \eqref{star} has a solution. Indeed, if $\beta^2 - 4\alpha = 3 V^{4/3}$, then $k = 3$ and $y = 1$
is a solution of \eqref{star}. A direct substitution yields that $f > 0$ in such point if and only if $e > \frac{3}{2}$. 
Also, if $g > 0$, then the right hand side of \eqref{star} is of order $g/y^3$ and the left hand side is of order $1/y^2$ around zero. Let $y_0$ be the smallest positive root of the right hand side of \eqref{star}
and notice that $y_0 \leq 1$. Since the left hand side is always non-negative, by the intermediate value theorem, there is a solution $y^*$ of \eqref{star}. In addition, since $y_0 \leq 1$ was the first zero, 
we have that $f(w^*) \geq 0$ with $w^*$ corresponding to $y^*$, and consequently there is a point on $M$ with vanishing curvature. 

Finally notice that \eqref{star2} is equivalent to a polynomial of at most fourth degree, and therefore for any set of parameters, the set of points on $M$ with one principal curvature vanishing, 
consists of a union of at most four $N - 2$
dimensional surfaces. 
\end{proof}

Proposition \eqref{propM} allows us to use the following Fourier restriction result proved in \cite{MR155146} and \cite{Herz}. We refer to \cite{Herz} if $k=N-1$ and \cite{MR155146} if $f \equiv 1$.

\begin{thm}
\label{restthm}
Let $M\subset \R^N$ be a smooth compact and closed hypersurface with $k$ principal curvatures bounded away from $0$. Then, for any smooth $f$ in a neighbourhood of $M$, one has 
$$\left|\int_{M} e^{ix \xi} f(\xi) d\mathcal{H}^{N-1}(\xi) \right| \leq  C (1+|x|)^{-\frac{k}{2}}.$$
\end{thm}

Instead of considering the integral over $\R^N$, we restrict the integration domain to 
$M_\tau :=\{\xi \in \R^N : F(\xi) =P(|\xi|)-\xi V=\tau \} $, for $|\tau|\leq \rho$ sufficiently small. Specifically, we 
define a cut-off function $\chi \in C^\infty_0 (\R)$ such that $0\leq \chi \leq 1$, $\textrm{supp} (\chi) \subset [-\rho, \rho] $ and $\chi \equiv 1$ on $(-\rho/2, \rho/2)$.
Define $\tilde{\chi}(\xi) = \chi(F(\xi))$ and consider
separately 
\begin{align}
G_{\varepsilon}^1 (x) &= \int_{\R^N} \dfrac{1 - \tilde{\chi} (\xi)}{P(|\xi |) - \xi V -i\varepsilon} e^{ix\xi} d\xi \\
G_{\varepsilon}^2 (x) &= \int_{\R^N} \dfrac{\tilde{\chi} (\xi)}{P(|\xi |) - \xi V -i\varepsilon} e^{ix\xi} d\xi. 
\end{align}
First, we observe that $G_{\varepsilon}^1$ is the Fourier transform of a function $Q$, with the property that $D^\gamma Q \in L^1 (\R^N)$ for any sufficiently large $|\gamma|$, where $\gamma$ is a multi-index. 
Thus, $|G_{\varepsilon}^1 (x)| \leq C_\gamma |x|^{-|\gamma|}$ for large $|\gamma|$, and therefore $G_{\varepsilon}^1$ has arbitrarily (polynomial) fast decay at infinity. 
To estimate $G_{\varepsilon}^2$ we use the coarea formula 
\begin{equation}
\label{greenint}
G_\varepsilon^2 (\xi) = \int_\R \dfrac{\chi (\tau )}{\tau - i\varepsilon} \displaystyle\left( \int_{M_\tau} \dfrac{e^{ix \xi}}{|\nabla F(\xi )|} d\mathcal{H}^{N-1}(\xi) \right) d\tau=
\int_{-\rho}^\rho \dfrac{\chi (\tau )}{\tau - i\varepsilon}a_x (\tau) d\tau , 
\end{equation}
where 
$$a_x (\tau)=\int_{M_\tau} \dfrac{e^{ix \xi}}{|\nabla F(\xi )|} d\mathcal{H}^{N-1}(\xi)  .$$  
Then, (A2) yields $\nabla F(\xi) \neq 0$ for $\xi \in M_0$. Therefore by continuity, decreasing $\rho > 0$ if necessary, we can assume that $\nabla F(\xi) \neq 0$ for any $\xi \in M_0$ and $|\tau| < \rho$.  
Since $F$ is smooth, we have in addition  $\left\|\dfrac{1}{\nabla F(\xi )} \right\|_{C^{\gamma,\tilde{\alpha}}(U) }\leq C$, for all $\gamma >0$ and $\tilde{\alpha}\in (0,1)$. 
Hence, from Proposition \ref{propM} and Theorem \ref{restthm}, it follows that
\begin{equation}
\label{phase0}
|a_x (0)| \leq C (1+|x|)^{-\frac{k}{2}}, 
\end{equation}
with 
\begin{equation}
\label{deftildeN}
k =\begin{cases}
N-1,\ &\textrm{if }\ \eqref{assumption1}\ \textrm{holds},\\ 
N-2,\ &\textrm{if } \ \eqref{assumption2}\ \textrm{holds}.
\end{cases}
\end{equation}
 Proceeding for instance as in \cite{MR3949725} (see \cite[(43)]{MR3949725} ), one can also show that if $t>0$ is small enough, then for some $b>0$,
\begin{equation}
\label{phase1}
|a_x (t) - a_x (0)| \leq C t^{b} (1+|x|)^{-k /2}.
\end{equation}
Observe that we already obtained in \eqref{phase0} the decay of $a_x$ as a function of $x$. The H\" older continuity in time $t$ follows by trivial modifications from \cite{MR3949725}. 

 Next, we to use the following proposition
\begin{prop}  \cite[Proposition 7]{MR3949725}
\label{mandel7}
Fix $\lambda \in \R$ and $\rho \in (0,\infty]$ and assume that $a:[\lambda -\rho ,\lambda +\rho] \rightarrow \R$ is measurable such that $|a(\lambda +t) - a(\lambda)| \leq \omega (|t|)$ where $t\rightarrow \omega (t) /t$ is integrable over $(0,\rho)$. Then, for any $\varepsilon >0$,
\begin{align*}
&\left|\int_{\lambda - \rho}^{\lambda +\rho} \frac{a(\tau)}{\tau -\lambda \mp i\varepsilon} d\tau - p.v. \int_{\lambda -\rho}^{\lambda +\rho} \dfrac{a(\tau)}{\tau -\lambda} d\tau \mp i\pi a (\lambda )\right|\\
&\leq \int_0^\rho \dfrac{2\varepsilon}{\sqrt{t^2 +\varepsilon^2}} \dfrac{\omega (t)}{t} dt + \left(\pi - 2 \arctan \left(\frac{\rho}{\varepsilon}\right)\right) |a(\lambda )|,
\end{align*}
and
$$\left|\int_{\lambda -\rho}^{\lambda +\rho} \dfrac{a(\tau)}{\tau -\lambda \mp i\varepsilon} d\tau \right| \leq 2\pi \left(\int_0^\rho \dfrac{\omega (t)}{t}dt +|a(\lambda)| \right). $$
\end{prop}

We have all ingredients to estimate the decay of $G$.

\begin{prop}
\label{asymgreen}
Assume that (A1), (A2) and (A3) hold. Let $k$ be as in \eqref{deftildeN}, that is, $k$ is the number of principal curvatures of $M$ staying bounded away from $0$. Then,  for $|x|\geq 1$,
\begin{equation}
\label{greenzero1}
|G(x)|\leq C (1+|x|)^{-k /2},
\end{equation}
and, for $|x|\leq 1$,
\begin{equation}
\label{greenzero}
|G(x)|\leq C\begin{cases} |x|^{-(N-4)},& \textrm{if}\ N>4,\\ |\log |x||,& \textrm{if} \ N=4,\\ 1,& \textrm{if}\ N<4 .\end{cases}
\end{equation}

\end{prop}

\begin{proof}
The estimate \eqref{greenzero} is classical (see for example \cite[Theorem 5.7]{Tanabe1997}) so we focus on \eqref{greenzero1}, where we follow the  framework from  \cite[Proof of Proposition $3$]{MR3949725}.

By \eqref{phase1}, we obtain that $\chi a_x$ satisfies 
 $|(\chi a_x) (t) - (\chi a_x) (0)|\leq  \omega_x (t)= Ct^b (1+|x|)^{-k /2}$ with integrable $t\rightarrow \omega_x (t) /t$ on $(0,\rho)$.
 If we define
$$
G_\pm^2 (x) = p.v. \int_{-\rho}^\rho \dfrac{\chi (\tau )}{\tau } a_x (\tau) d\tau  \pm i\pi a_x (0) 
\,,
$$
then Proposition \ref{mandel7} with $a = \chi a_x$, $\varepsilon = 1$, $\lambda = 0$ and \eqref{phase0}, yield 
\begin{align*}
|G^2_\pm (x)| &\leq\left| G^2_\pm (x) - \int_{- \rho}^{\rho} \frac{a(\tau)}{\tau  \mp i} d\tau\right| + 
\left| \int_{ - \rho}^{\rho} \frac{a(\tau)}{\tau  \mp i} d\tau\right|\\
& \leq C \left(\int_0^\rho \dfrac{\omega_x (t) }{t} dt+|a_x (0|\right)\\
&\leq C  (1+|x|)^{-k /2} \int_0^\rho t^{b -1} dt +C (1+|x|)^{-k /2}\\
&\leq C (1+|x|)^{-k /2}.
\end{align*}
Also, assuming that $\varepsilon >0$, the first estimate of Proposition \ref{mandel7}, and  \eqref{phase0} imply
\begin{align*}
|G^2_\varepsilon (x) - G^2_+ (x)|& \leq \int_0^\rho \dfrac{2\varepsilon}{\sqrt{\varepsilon^2 +t^2}} \dfrac{\omega_x (t)}{t}dt  + \left(\pi -2\arctan\left(\frac{\rho}{\varepsilon}\right)\right)|a_x (0)|\\
&\leq C (1+|x|)^{-k /2} \int_0^\rho \dfrac{\varepsilon}{\sqrt{\varepsilon^2 +t^2}} t^{b -1} dt +C (1+|x|)^{-k /2} \dfrac{2\varepsilon}{\rho}\\
&\leq C \varepsilon^b (1+|x|)^{-k /2}
\end{align*}
and analogous inequality for $\varepsilon <0$ after replacing $G^2_+$ by $G_-^2$. 
Therefore, since $b>0$, we deduce that $G^2_\varepsilon$ converges pointwise to $G^2_\pm$  when $\varepsilon \rightarrow 0^\pm$. 

Overall, we get
\begin{align}
|G(x)| &= \limsup_{\epsilon \to 0} |G_\epsilon(x)| \leq \limsup_{\epsilon \to 0} |G^1_\epsilon(x)| +  |G^2_\epsilon(x)| \\
&\leq  
\limsup_{\epsilon \to 0^\pm} |G^1_\epsilon(x)| +  |G^2_\epsilon(x) - G^2_\pm| + |G^2_\pm| \leq  C (1+|x|)^{-k /2} \,,
\end{align}
as desired. 
\end{proof}

\section{Resolvent estimate and application to the nonlinear Helmholtz equation}

In this section we investigate the boundedness of the resolvent $\mathfrak R "=" (\Delta^2 +\beta \Delta +iV \nabla +\alpha)^{-1}$ as an operator from $L^p (\R^N)$ to $L^q (\R^N )$. 
The obtained estimates are crucially used in the construction of solution to the nonlinear Helmholtz equation. 
As in previous sections, we define $\mathfrak R$, by the limit absorption principle, that is, 
\begin{equation}
\label{eq:def_resolvent}
\mathfrak R f:= \lim_{\varepsilon \rightarrow 0^+} \mathfrak R_{i\varepsilon} f,
\end{equation}
 where
$$ (\mathfrak R_{i\varepsilon} f)(x):=\dfrac{1}{(2\pi)^{N/2}} \int_{\R^N} e^{ix\xi} \dfrac{\hat{f} (\xi )}{|\xi|^4 -\beta |\xi|^2 + V \xi +\alpha -i\varepsilon}.  $$
To study the mapping properties of $\mathfrak R$, we use ideas from \cite[Proof of Theorem~6]{Gu}
 (see also~\cite[Theorem~2.1]{Ev} and  \cite[Theorem $3.3$]{MR3977892}). Our proof relies on the decay properties of the Green function
  established in the previous section and on the Stein-Tomas Theorem, which was proved in  \cite{MR2831841} and improves  \cite{MR1810754, MR1882456}. 
 We also refer to \cite{MR3556359} for an example proving that the used results are sharp in some sense.

\begin{thm}\cite{MR2831841} \label{STT}
  Let $1\leq p\leq \frac{2(k+2)}{k+4}$ and $M$ be a smooth, closed and compact manifold with $k$ non-null principal curvatures. Then, there is a $C>0$ such that for all $g\in \mathcal S(\R^N)$ the
  following inequality holds: 
  \begin{equation*} 
    \Big(\int_{M} |\hat{g}(\omega)|^2 \,d\sigma(\omega)\Big)^{1/2} 
    \leq C \|g\|_{L^p(\R^N)}.
  \end{equation*} 
\label{STthm}
\end{thm}

Before we proceed, let us introduce some notation. 
Let $\psi \in \mathcal{S}(\R^N)$ be a function such that $\hat{\psi}\in C_c^\infty(\R^N)$ satisfies
$0\leq \hat{\psi}\leq 1$ and
\begin{equation} \label{eq:def_psi}
  \hat{\psi} (\xi)=\begin{cases}
    1 ,\ \textrm{if}\ \textrm{dist}(\xi , M )\leq c_1,\\ 
    0 ,\ \textrm{if}\ \textrm{dist}(\xi , M)\geq 2c_1,
    \end{cases}
\end{equation}
for some $c_1$ small enough detailed below. As above, define $G_1 := (1 - \psi) \ast G$ and $G_2:= \psi \ast G$ and $\mathfrak{R}_i f = G_i*f$ for $i = 1, 2$.  
First, we obtain estimates for $\mathfrak{R}_1$.

\begin{prop}
For every $c_1 > 0$, any $p, q \in (1, \infty)$ with $q > p$ and
 \begin{align}
  \label{condgutie}
	\frac{1}{p}-\frac{1}{q} 
	\begin{cases}
	  \leq 1 &\textrm{if } N<4,\\
	  < 1 &\textrm{if } N=4,\\
	  \leq \frac{4}{N} &\textrm{if } N\geq 4,
	\end{cases}
	\end{align}
 the operator $\mathfrak{R}_1$ extends to a
  bounded linear operator $\mathfrak{R}:L^p(\R^N)\to L^q(\R^N)$, that is, there exists $C = C_{p, q}$ such that 
  \begin{equation} 
    \|\mathfrak{R} f \|_{L^q(\R^N)}
    \leq C \|f\|_{L^p (\R^N)} \,.
  \end{equation}
\end{prop}

\begin{proof}
By Proposition \ref{asymgreen}, we have
\begin{equation}
\label{estG2}
|G_1 (x)|\leq \begin{cases} C \min \{|x|^{4-N},  |x|^{-N}\} &\text{if}\ N> 4,\\ 
  C \min \{1+ |\log|x||, |x|^{-N}\} &\text{if}\ N= 4,\\
C \min \{1,|x|^{-N} \}& \text{if}\ N<4, \\
\end{cases}
  \qquad\text{for all }x\in\R^N \setminus \{0\}  \,.
\end{equation}
Fix $p,q \in (1, \infty)$ as in \eqref{condgutie} and  assume $\frac{1}{p} - \frac{1}{q} < \frac{4}{N}$ if $N \geq 4$. Set $\frac{1}{r} = 1 + \frac{1}{q} - \frac{1}{p}$. 
Since $q > p$,  
then  $r\in \left(1, \frac{N}{(N-4)_+}\right)$, where we set $a/0 = \infty$. 
Since $G_1$  belongs to the weak Lebsegue space $L^{r, w}(\R^N)$,  the Young's convolution inequality
for weak Lebesgue spaces, (see ~\cite[Theorem~1.4.24]{Grafakos})  gives us
\begin{equation}\label{RT0}
  \|G_1 \ast f \|_{L^q(\R^N )} \leq \|G_1\|_{L^{r, w} (\R^N )} \|f\|_{L^p(\R^N )} \leq C \|f\|_{L^p(\R^N )} \,,
\end{equation}
as desired. If $N > 4$, and $\frac{1}{p} - \frac{1}{q} = \frac{4}{N}$, then $r < \infty$ and we proceed as above. 
\end{proof}

In the next result we establish the crucial bound on $\mathfrak R_2$.
Since $\psi$ in the definition of $G_2$ is a Schwartz function such that $\hat{\psi} \in C^\infty_c(\R^N)$, then its convolution yields a bounded function, and 
 by Proposition \ref{asymgreen},
\begin{equation} \label{gutiee1}
  |G_2 (x)|\leq C (1+|x|)^{-\frac{k}{2}} \quad\text{for all }x\in\R^N \,.
\end{equation}
Let $\eta \in C_{c}^\infty (\R^N) $ be a cut-off
function such that $\eta(x)=1$ for $|x| \leq 1$ and $\eta(x)=0$ if $|x|\geq 2$. For $j\in \N$ we define
$\eta_j(x):=\eta (x/2^j ) - \eta (x /2^{j-1})$ and $\eta_0 :=\eta$. Therefore, by \eqref{gutiee1},
\begin{equation} \label{eq:def_G1j}
  G_2=\sum_{j=0}^\infty G_2^j\quad \text{with}\; G_2^j:=G_2 \eta_j \text{ so that }
  |G_2^j(x)| \leq C 2^{-jk/2}1_{[2^{j-1},2^{j+1}]}(|x|),
\end{equation}
and
 \begin{equation}\label{eq:def_Qj}
   G_2\ast f =  \sum_{j=0}^\infty G_2^j\ast f \,. 
 \end{equation}

\begin{thm}  \label{thminvopbound}
  Assume that (A1), (A2), and (A3) hold and $k$ is as in Theorem \ref{STT}. Then, for any sufficiently small $c_1 > 0$ there is $C > 0$ such that 
  \begin{equation} \label{eq:resolvent_estimate}
    \|G_2^j* f \|_{L^2(\R^N)}
    \leq C 2^{j/2} \|f\|_{L^{2(k +2)/(k+4)}(\R^N)} ,
  \end{equation}
  for any $j \geq 1$ and $f \in L^{2(k +2)/(k+4)}(\R^N)$.
\end{thm}

\begin{proof}
 Using  Plancherel's Theorem and the coarea formula, we obtain, for $f\in\mathcal{S}(\R^N)$,
\begin{align*}
  \begin{aligned}
  \|G^j_2 \ast f\|_{L^2 (\R^N )}^2
  &\leq C \int_{-2c_1}^{2c_1}  \max_{\xi \in M_\tau}(|\hat{G}_2^j (\xi )|^2) \int_{M_\tau}
\dfrac{  |\hat{f}( \xi )|^2}{|\nabla F(\xi )|} d\sigma(\xi) \,d\tau .
  \end{aligned}
\end{align*}

Since $F$ is a smooth function with $\nabla F \neq 0$ on $M_0$, then by making $c_1$ smaller if necessary, we may assume $\nabla F \neq 0$ on $M_\tau$ for $|\tau| < 3c_1$. Then,  
$\xi \mapsto |\nabla F(\xi )|^{-1}$ is also smooth on $M_\tau$ for $|\tau| < 3c_1$ and since 
$M_\tau$, for $|\tau|$ sufficiently small, is smooth closed and compact manifold with $k$ principal curvatures bounded away from $0$, the Stein-Tomas Theorem (see Theorem~\ref{STT}) implies
\begin{align*} 
  \begin{aligned}
  \|G_2^j \ast f\|_{L^2 (\R^N )}^2
  &\leq C \|f\|_{L^{ 2(k+2)/(k+4)}(\R^N )}^2 \int_{-2c_1}^{2c_1}  \max_{\xi \in M_\tau}(|\hat{G}_2^j (\xi )|^2) d\tau. 
  \end{aligned}
\end{align*}
The coarea formula, yields
\begin{align}
\hat{G}_2^j (\xi)&= 2^{(j-1)N}\int_{-2c_1}^{2c_1} \frac{1}{\tau} \int_{M_\tau}  \frac{\hat{\psi} (s)  \hat{\eta}_1 (2^{j-1} (\xi - s))  }{|\nabla F (s)|}d\sigma (s) d\tau \,.
\end{align}
Let us define the map $\phi_\tau : M_\tau \rightarrow M_0$ such that $\phi_\tau (w)=w+ \tau m(\tau) \nu (w)$, where $\nu(w)$ is the unit exterior normal vector to $M_0$ at $w$
and $m$ is the smallest number such that $\phi_\tau(w) \in M_\tau$. Since $M_\tau$ is a level surface of $F$, then $\nabla F(w) = \nu(w)$, and by Taylor theorem,  for any sufficiently small $c_1$,  
there is a constant $c_2$ such that $c_2^{-1} \leq m(\tau) \leq c_2$ and $|m' (\tau)| \leq c_2$ for any $|\tau| < 3c_1$. Thus, $\phi_\tau$ is a smooth bijection, and therefore we can introduce the change 
of variables 
 $s= \phi_\tau (w)$ and obtain
$$\hat{G}_2^j (\xi) =
2^{(j-1)N}\int_{-2c_1}^{2c_1} \frac{1}{\tau} \int_{M_0}  \frac{\hat{\psi} (\phi_\tau(w))  \hat{\eta}_1 (2^{j-1} (\xi - \phi_\tau(w)))  }{|\nabla F (\phi_\tau(w))|}|J_{\tau} (w) | d\sigma (w) d\tau\,,
$$
where $|J_{\tau} (w)|$ is the Jacobian of the transformation $\phi_\tau$. 
Using the oddness of the function $\tau \rightarrow 1/\tau$, we have
\begin{align*}
\hat{G}_2^j (\xi) = 2^{(j-1)N}\int_{0}^{2c_1} \frac{1}{\tau} \int_{M_0} & (\frac{\hat{\psi} (\phi_\tau (w))  \hat{\eta}_1 (2^{j-1} (\xi - \phi_\tau (w) ))  }{|\nabla F (\phi_\tau (w))|}| J_\tau (w) | \\
&- \frac{\hat{\psi} (\phi_{-\tau} (w))  \hat{\eta}_1 (2^{j-1} (\xi -\phi_{-\tau} (w) ))  }{|\nabla F (\phi_{-\tau} (w))|}|J_{-\tau} (w) | d\sigma (w) d\tau.
\end{align*}
Next, write $\hat{G}_2^j$ as
$$\hat{G}_2^j (\xi)= 2^{(j-1)N}\int_{0}^{2c_1} \frac{1}{\tau} \int_{M_0}  (I+II+III+IV)d\sigma (w) d\tau, $$
where, after omitting the argument $w$,
$$I=\frac{(\hat{\psi} (\phi_\tau)-\hat{\psi} (\phi_{-\tau}) )  \hat{\eta}_1 (2^{j-1} (\xi -\phi_\tau  ))  }{|\nabla F (\phi_\tau )|}|J_\tau | ,$$
$$II= \frac{\hat{\psi} (\phi_{-\tau}  )  ( \hat{\eta}_1 (2^{j-1} (\xi -\phi_\tau   )) - \hat{\eta}_1 (2^{j-1} (\xi -\phi_{-\tau}   ))  )  }{|\nabla F (\phi_\tau  )|}|J_\tau  | ,$$
$$III=  (\hat{\psi} (\phi_{-\tau}  )  \hat{\eta}_1 (2^{j-1} (\xi -\phi_{-\tau}   ))   ) \left(\dfrac{1}{|\nabla F (\phi_\tau  )|}-\dfrac{1}{|\nabla F (\phi_{-\tau}  )| }\right)|J_\tau  | ,$$
and
$$IV= \frac{\hat{\psi} (\phi_{-\tau}  )  \hat{\eta}_1 (2^{j-1} (\xi -\phi_{-\tau}  ))    }{|\nabla F (\phi_{-\tau}  )| }( |J_\tau  |- |J_{-\tau}  | ).$$
First we estimate $I$. Since $\hat{\psi}$ is smooth and $m$ is bounded, by the mean value theorem 
$$
|\hat{\psi} (\phi_\tau (w))-\hat{\psi} (\phi_{-\tau} (w))| \leq C |\phi_\tau (w) - \phi_{-\tau} (w)| = C |\tau m(\tau) + \tau m(-\tau)| \leq C \tau \,, 
$$
and consequently the boundedness of $|\nabla F|^{-1}$ and a return to the original variables imply
\begin{multline}
 2^{(j-1)N}\int_{0}^{2c_1} \frac{1}{\tau} \int_{M_0} I d\sigma (w) d\tau \\
\begin{aligned} 
 &\leq C 2^{jN} \int_0^{2c_1} \int_{M_0} | \hat{\eta}_1 (2^{j-1} (\xi -\phi_\tau (w) ))|  \dfrac{|J_\tau (w)|}{|\nabla F (\phi_{\tau} (w))|} d\sigma (w) d\tau \\
&= C 2^{jN} \int_0^{2c_1}  \int_{M_\tau} \dfrac{| \hat{\eta}_1 (2^{j-1} (\xi -s))|}{|\nabla F (s)|} d\sigma (s) d\tau\\
&\leq C 2^{jN} \int_{\R^N} | \hat{\eta}_1 (2^{j-1} (\xi -z)) |dz = C 2^{jN} \int_{\R^N} | \hat{\eta}_1 (2^{j-1} z) |dz\\
&\leq C \,,
\end{aligned}
\end{multline}
where in the last step we used that $\hat{\eta}_1$ is a Schwartz function and in particular integrable.

To estimate $III$ we use that 
$|\nabla F|^{-1}$ is smooth, and therefore we get
$$\left| \dfrac{1}{|\nabla F (\phi_\tau (w))|}-\dfrac{1}{|\nabla F (\phi_{-\tau} (w))| }\right|\leq C \tau ,$$
and the estimate for $III$ follows as above.

Also, since $\phi$ is smooth and non-degenerate, then $J_\tau$ (proportional to $\nabla \phi$) is smooth and $|J_\tau (w)| \geq c > 0$ for any $|\tau| \leq 3c_1$. Therefore, 
$$||J_\tau (w)|- |J_{-\tau} (w)| |\leq C \tau |J_\tau (w)|$$
and the estimate for $IV$ follows as above.

Finally, for $II$, by the mean value theorem, we have
\begin{align*}
\hat{\eta}_1 (2^{j-1} (\xi -\phi_\tau (w) )) &- \hat{\eta}_1 (2^{j-1} (\xi -\phi_{-\tau} (w) )) = \int_{-1}^1 \frac{d}{dr}  \hat{\eta}_1 (2^{j-1} (\xi -\phi_{r\tau} (w) )) dr \\
&= \int_{-1}^1  \nabla \hat{\eta}_1 (2^{j-1} (\xi -\phi_{r\tau} (w) ))\nu(w) (\tau m(r\tau) + r\tau^2m'(r\tau)) dr \,.
\end{align*}
Then, since $|m|, |m'| \leq C$ and $|J_{\tau r}| \geq c > 0$ for $|\tau|\leq 2c_1$ and any $|r| \leq 1$, by following steps above
we obtain 
\begin{multline*}
2^{(j-1)N}\int_{0}^{2c_1} \frac{1}{\tau} \int_{M_0} IV d\sigma (w) d\tau  \\
\begin{aligned}
&\leq C 2^{jN} \int_{-1}^1 \int_{0}^{2c_1} \int_{M_0}  |\nabla \hat{\eta}_1 (2^{j-1} (\xi - \phi_{r\tau} (w) )) | d\sigma (w) d\tau  dr \\
&\leq C 2^{jN}  \int_{-1}^1 \int_{0}^{2c_1} \int_{M_0} |\nabla \hat{\eta}_1 (2^{j-1} (\xi -\phi_{r\tau} (w) ))| |J_{\tau r}(w)| d\sigma(w)d\tau dr \\
&= C 2^{jN}  \int_{-1}^1 \int_{0}^{2c_1} \int_{M_{\tau r}} |\nabla \hat{\eta}_1 (2^{j-1} (\xi - s ))|  d\sigma(s)d\tau dr \\
&\leq C2^{jN}  \int_{-1}^1 \int_{\R^N}  |\nabla \hat{\eta}_1 (2^{j-1} (\xi - z))|  dz dr \\
&\leq C2^{jN}  \int_{\R^N}  |\nabla \hat{\eta}_1 (2^{j-1} (\xi - z))|  dz
\end{aligned}
\end{multline*}
and the rest of the proof follows analogously as above with $\eta$ replaced by $\nabla \eta$.  

Overall, we showed
\begin{align} \label{RT2}
  \begin{aligned}
  \|G_2^j \ast f\|_{L^2 (\R^N )}^2
  &\leq C2^j \|f\|_{L^{ 2(k+2)/(k+4)}(\R^N )}^2 . 
  \end{aligned}
\end{align}
and the proof is finished. 
\end{proof}

\begin{thm}
  Assume that (A1), (A2), and (A3) hold and $k$ is as in Theorem \ref{STT}. Then, for any sufficiently small $c_1 > 0$ there is $C = C(p, q) > 0$ such that 
  \begin{equation} \label{eq:resolvent_estimate2}
    \|G_2* f \|_{L^q(\R^N)}
    \leq C \|f\|_{L^{p}(\R^N)} 
  \end{equation}
  provided $p$ and $q$ satisfy
\begin{equation}\label{terr}
\begin{aligned}
\frac{1}{p} - \frac{1}{q} &>  \frac{2}{2 + k} \,,\\
  \frac{1}{q} + \frac{(k+2)(N-k-1)}{kN} \frac{1}{p}  &> \frac{4N + 2kN - 4 - 6k - k^2}{2kN} \,,
  \\
  \frac{(k+2)(N - 1 - k)}{kN} \frac{1}{q} + \frac{1}{p}  &>  1-\frac{k}{2N} \,.
\end{aligned}
\end{equation}
If $p = q'$, then \eqref{terr} reduces to  $q > \frac{2(k + 2)}{k}$. 
\end{thm}

\begin{proof}
Using \eqref{eq:def_G1j}, for any  $\tilde{p}, \tilde{q}, r \in [1,\infty ]$ such that $1+\frac{1}{\tilde q} = \frac{1}{r}+\frac{1}{\tilde p}$, Young inequality gives
\begin{align}  \label{RT1}
  \begin{aligned}
  \|G^j_2 \ast f\|_{L^{\tilde q}(\R^N)}
  &\leq \|G_2^j\|_{L^r(\R^N )} \|f\|_{L^{\tilde p}(\R^N)} \\ 
  &\leq C 2^{j\left(- \frac{k}{2} + \frac{N}{r}\right)}\|f\|_{L^{\tilde p}(\R^N)} \\
  &= C 2^{j\left(-\frac{k}{2} + N +\frac{N}{\tilde q} - \frac{N}{\tilde p}\right)}\|f\|_{L^{\tilde p}(\R^N)} \,.
  \end{aligned}
\end{align}
Fix $q\geq 2$, $1\leq p\leq \frac{2(k+2)}{k+4}$, and $\tilde{p} \leq p$, $\tilde{q} \geq q$, and $\theta \in [0, 1]$ such that 
\begin{equation}
\frac{1}{p}=\frac{\theta}{\tilde p}+\frac	{(1-\theta)(k+4)}{2(k+2)} \,, \qquad \frac{1}{q}=\frac{\theta}{\tilde q}+\frac{1-\theta}{2} \,.
\end{equation}
By substituting $\tilde{p} = 1$ and $\tilde{q} = \infty$,  we observe that such choice is possible if and only if $1\geq \theta\geq \max\{1-\frac{2}{q},\frac{2(k+2)-p(k+4)}{k p}\}$. Then, 
from the Riesz-Thorin theorem, \eqref{eq:resolvent_estimate} and \eqref{RT1}, it follows that
$$
  \|G_2^j \ast f\|_{L^q(\R^N)}
  \leq C 2^{j\left(\frac{1}{2}+\theta N\left(1 - \frac{k+1}{2N}+\frac{1}{\tilde q} - \frac{1}{\tilde p}\right)\right)}\|f\|_{L^p(\R^N)}
  \qquad (j\in\Z), 
$$
which after substitution for $\tilde{p}$ and $\tilde{q}$ becomes
$$
  \|G_2^j \ast f\|_{L^q(\R^N)}
  \leq C 2^{j \left( \frac{2N+2+k}{2(k+2)} + \frac{N}{q} - \frac{N}{p} +\theta N \left(\frac{3}{2} - \frac{k+1}{2N} - \frac{k+4}{2(k+2)} \right)\right)}\|f\|_{L^p(\R^N)}
  \qquad (j\in\Z) \,.
$$
Denote 
\begin{equation}
A = \frac{2N+2+k}{2N(k+2)} \,, \qquad B = \left(\frac{3}{2} - \frac{k+1}{2N} - \frac{k+4}{2(k+2)} \right) \,.
\end{equation}
Since the factor of $\theta$ is positive, to make the exponent negative, we need to choose $\theta$ as small as possible and such $\theta$ exists if and only if  
\begin{equation}\label{ugl}
A + \frac{1}{q} - \frac{1}{p} + \left(1-\frac{2}{q}\right)B < 0, \, \textrm{and }
A + \frac{1}{q} - \frac{1}{p} + \left( \frac{2(k+2)}{kp} - \frac{k+4}{k}\right)B < 0\,.
\end{equation}
Then,  \eqref{ugl} guarantee that 
\begin{equation}\label{fle}
 \|G_2\ast f\|_{L^q(\R^N)}  \leq C \|f\|_{L^p(\R^N)} \,.
\end{equation}
To visualize \eqref{ugl} we substitute $x = 1/p$ and $y = 1/q$ and get
\begin{equation}
A + y - x + \left(1- 2y \right)B < 0, \, \textrm{and }
A + y - x + \left( \frac{2(k+2)}{k}x - \frac{k+4}{k}\right)B < 0.
\end{equation}
Therefore the admissible region $\Gamma$ in $xy$-plane lies between two lines and inside the square $[0, 1]\times [0, 1]$, see Figure \ref{fig2} for sample regions.

\begin{figure}
\centering 
\begin{subfigure}{.35\textwidth}
\centering
\includegraphics[width=5cm]{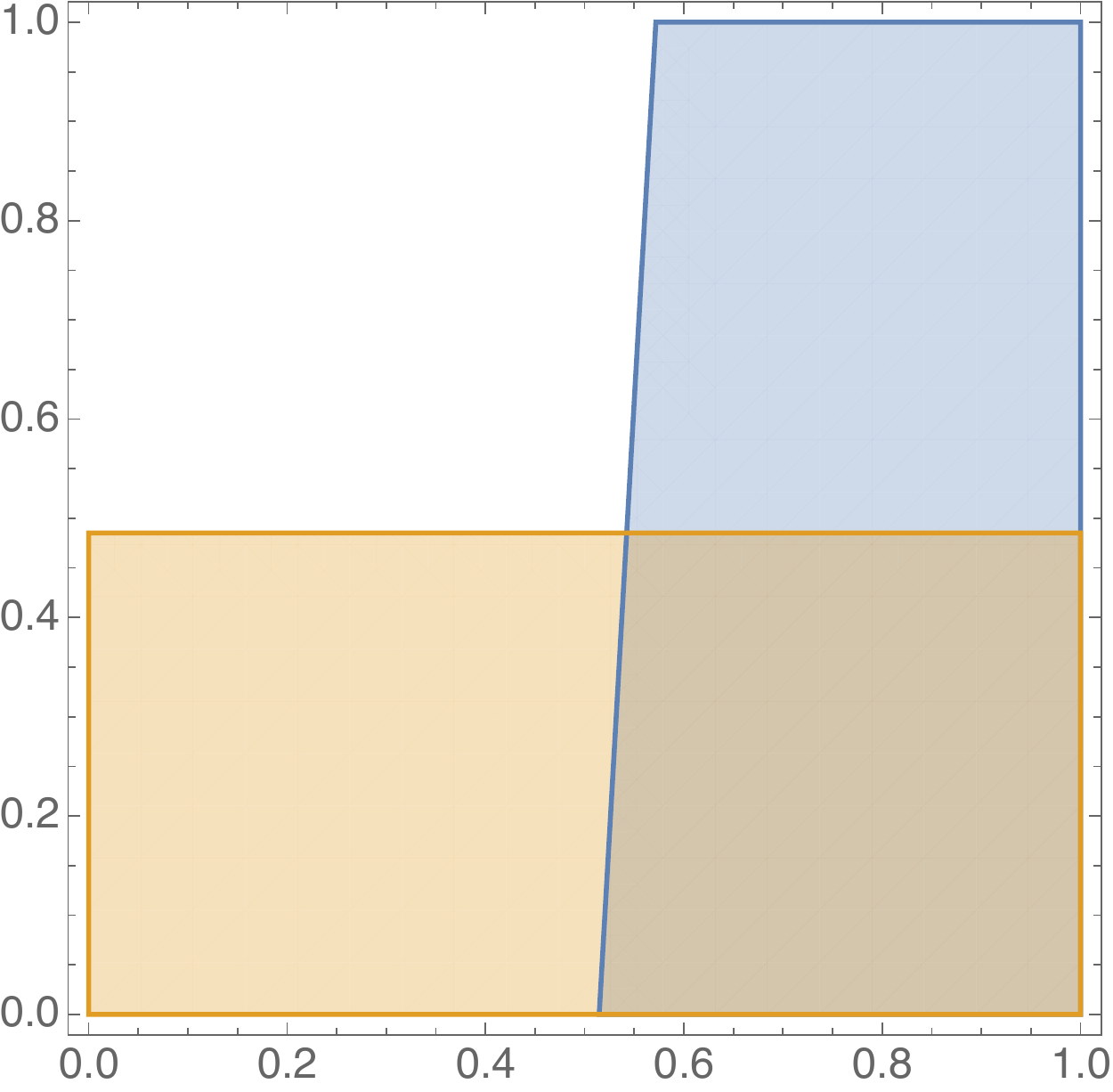}
\caption{$N = 34$, $k = 33$}
\end{subfigure}
\qquad\qquad 
\begin{subfigure}{.35\textwidth}
\centering
\includegraphics[width=5cm]{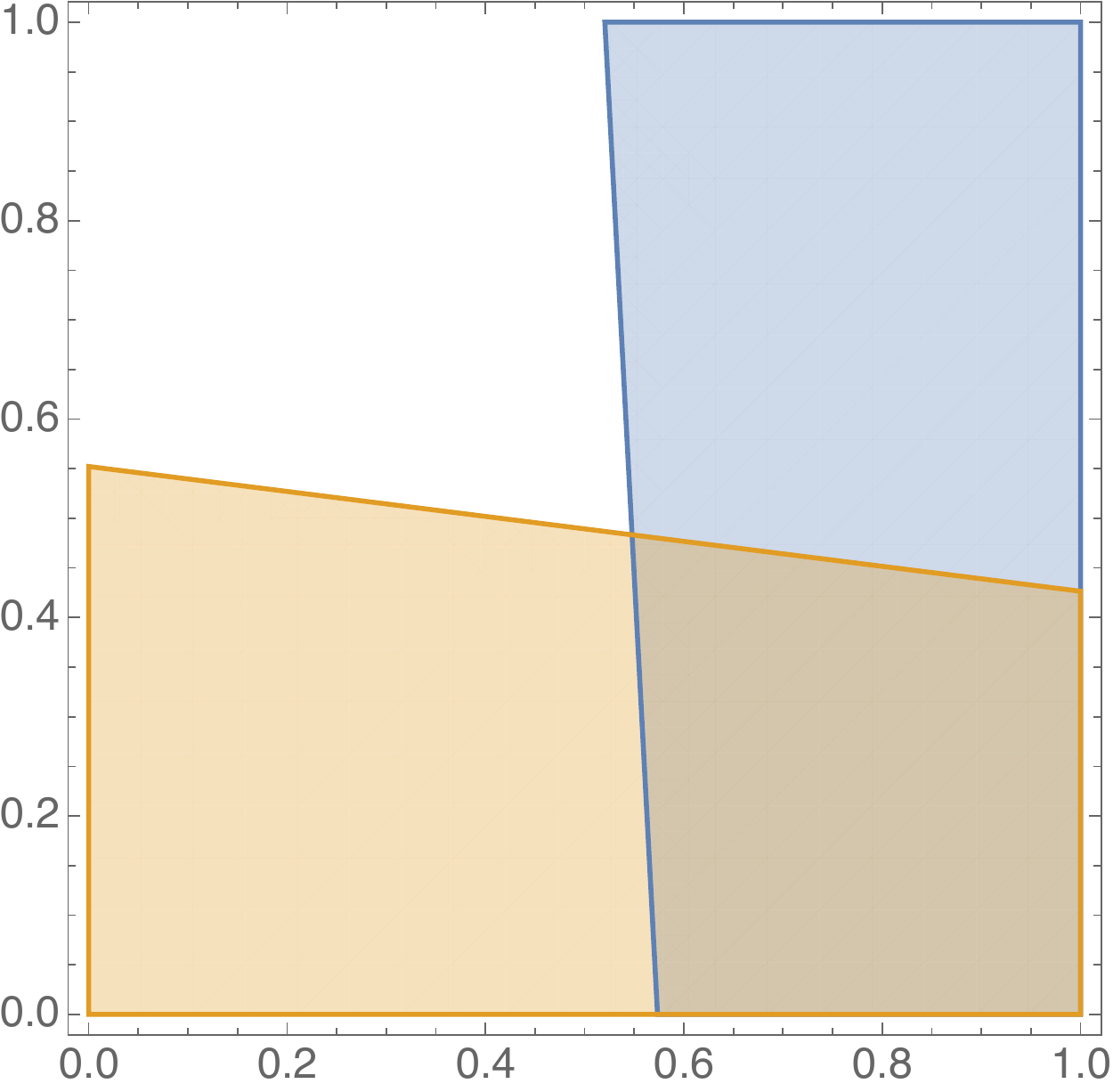}
\caption{$N = 34$, $k = 29$}
\end{subfigure}
\newline
\newline
\noindent
\centering
\begin{subfigure}{.35\textwidth}
\centering
\includegraphics[width=5cm]{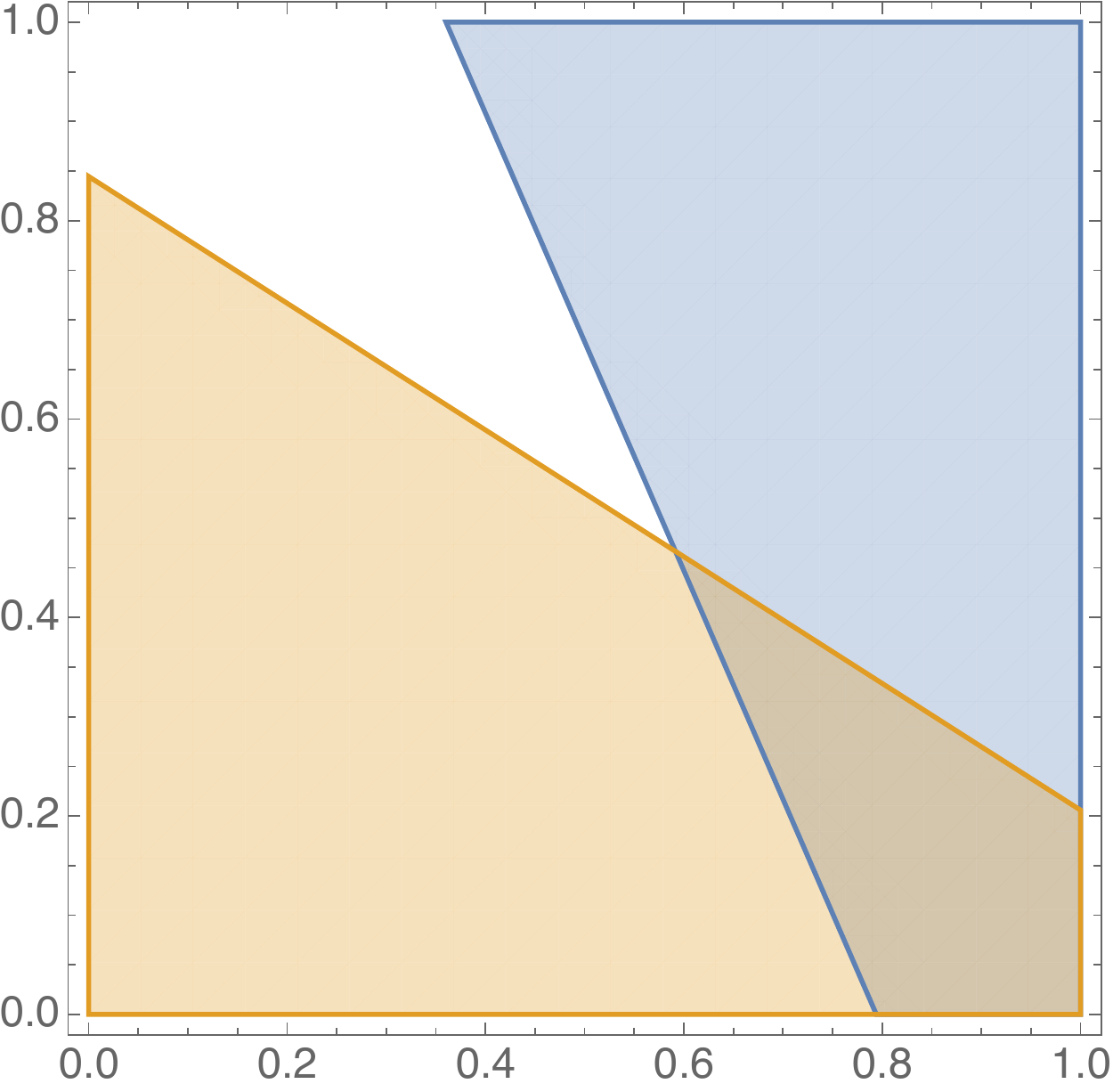}
\caption{$N = 34$, $k = 14$}
\end{subfigure}
\qquad \qquad 
\begin{subfigure}{.35\textwidth}
\centering
\includegraphics[width=5cm]{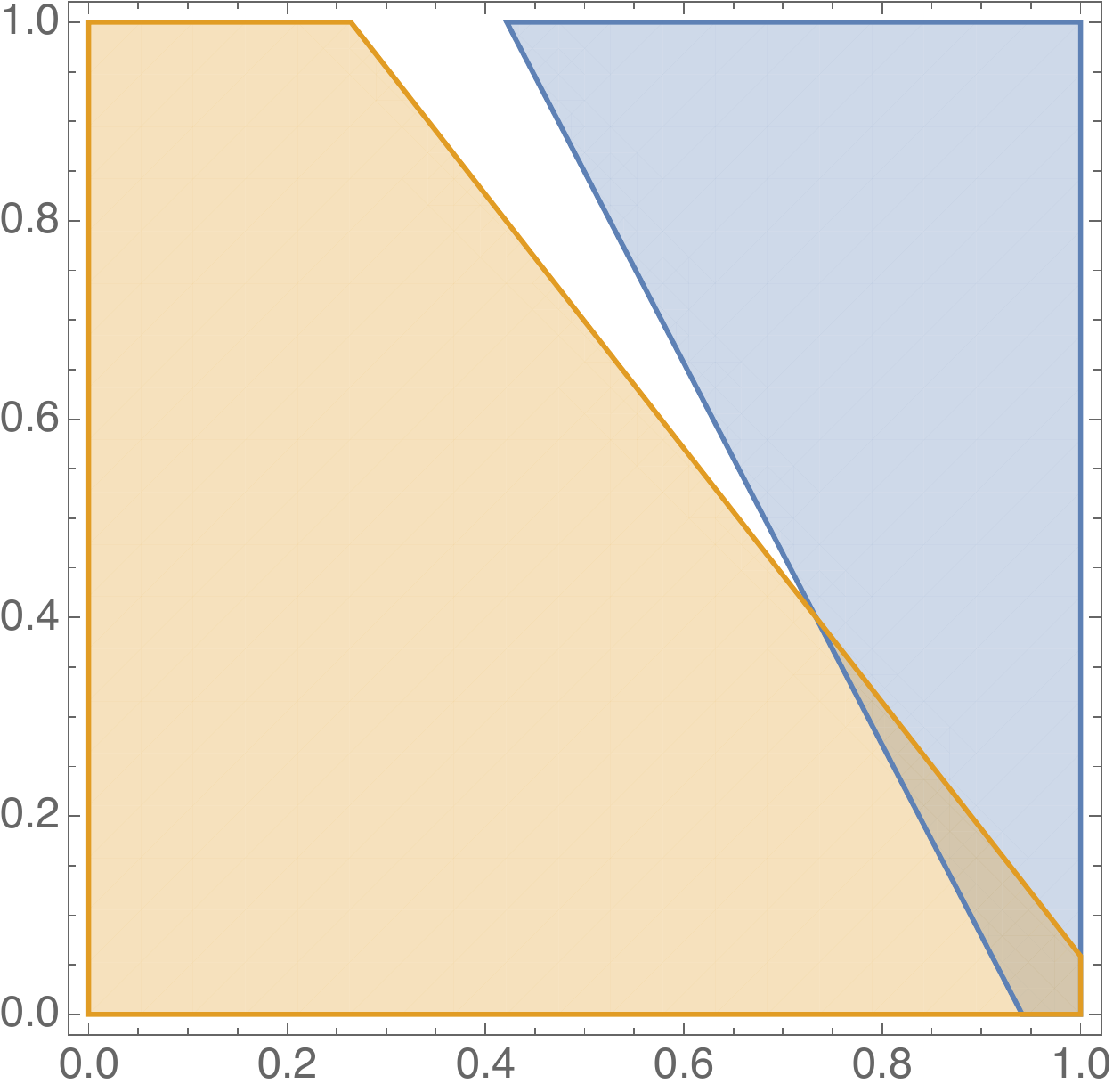}
\caption{$N = 34$, $k = 4$}
\end{subfigure}
\hphantom{1cm}
\caption{The region $\Gamma$ in dimension $N=34$ with selected values of $k$}\label{fig2}
\end{figure}

Since $\hat{G}_2$ is real, the convolution with $G_2$ is self-adjoint operator, and we obtain an estimate for the adjoint 
\begin{equation}
\|G_2\ast f\|_{L^{p'}(\R^N)}  \leq C \|f\|_{L^{q'}(\R^N)}
\end{equation}
whenever $p \leq \frac{2(k + 2)}{k+4}$ and $q \geq 2$ satisfy \eqref{ugl}. Noticing that if $(p, q)$ is admissible, then $(q', p')$ is admissible as well. 
Equivalently, if $(x, y)$ is admissible, then $(1 - y, 1 - x)$ is admissible as well. Thus, since $\Gamma$ is an admissible region, then the reflection of $\Gamma$, denoted $\Gamma'$,
with respect to the line $x + y = 1$ that lies inside the unit square is also an admissible region.  By the Riesz-Thorin theorem, we can interpolate, which in the $xy$-plane
means that the convex hull of $\Gamma \cup \Gamma'$ is also an admissible region. 

\begin{figure}
\centering
\includegraphics[width=5cm]{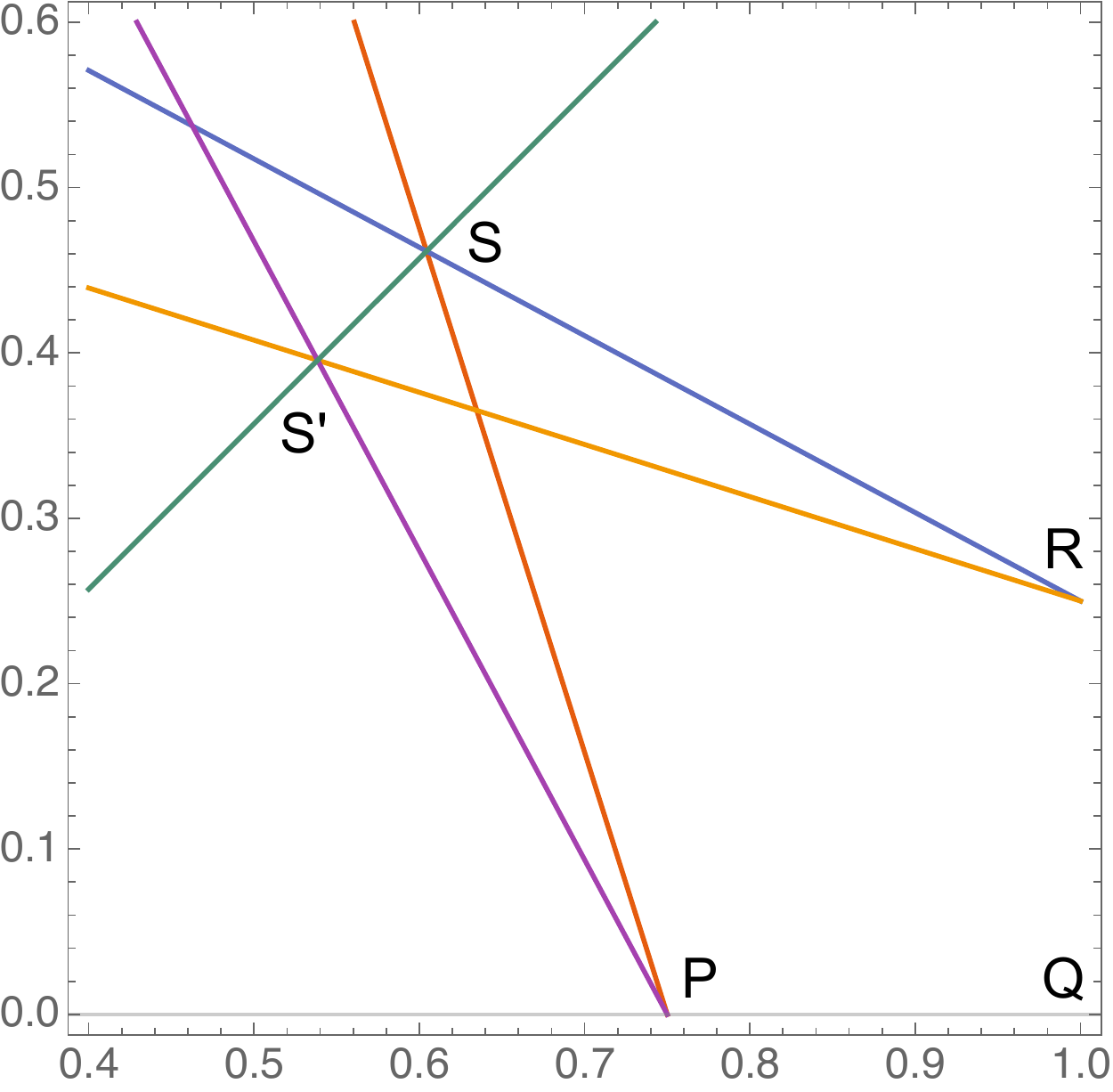}
\caption{Admissible region for $N = 24$ and $k = 12$}
\label{Fig3}
\end{figure}
Let us provide quantitative calculations. Using that $A + B = 1 - \frac{k}{2N}$ and calculating the intersections of lines, we obtain that he region $\Gamma$ is the quadrilateral (see Figure \ref{Fig3}) with the vertices 
\begin{equation}
P= \left(1 - \frac{k}{2N}, 0\right), \, Q = (1, 0), \,   R = \left(1, \frac{k}{2N}\right), \, S =  \left(\frac{4 + 6k + k^2}{4+6k + 2k^2}, \frac{k}{2(k + 1)}\right) \,.
\end{equation}
Perhaps surprisingly $S$ is independent of $N$. Since the $P$ and $R$ are symmetric with respect to the line $x + y = 1$, then $\Gamma'$
is a quadrilateral bounded by the vertices  $P, Q, R$ and 
\begin{equation}
S' = \left( \frac{k+2}{2(k+1)}, \frac{k^2}{4+6k + 2k^2}\right) \,.
\end{equation}
Finally, the convex hull of $\Gamma \cup \Gamma'$ is a pentagon with the vertices $P, Q, R, S, S'$. The sides of this pentagon lie on the lines 
\begin{gather}
y = 0,  \qquad x = 1, \qquad A + y - x + \left( \frac{2(k+2)}{k}x - \frac{k+4}{k}\right)B = 0\\
A  + y - x + \left( 1 - \frac{2(k+2)}{k}y \right)B = 0 , \qquad x - y =  \frac{2}{2 + k} \,.
\end{gather}
Transforming back to $p$ and $q$ and omitting trivial conditions $p \geq 1$ and $q \leq \infty$, we obtain that 
\begin{equation}
\|G_2\ast f\|_{L^q(\R^N)}  \leq C \|f\|_{L^p(\R^N)}
\end{equation}
if \eqref{terr} is valid. 

Finally, $p = q'$, is equivalent to $x = 1- y$, which is the axis of symmetry of $\Gamma \cup \Gamma'$. Thus, any point on axis of symmetry that lies inside the square $[0, 1]\times [0, 1]$ 
satisfying $x - y > \frac{2}{2 + k}$ belongs to the convex hull of $\Gamma \cup \Gamma'$, and the last assertion follows. 

\end{proof}

Finally, we use Theorem \ref{thminvopbound} to prove the existence of solution to the nonlinear Helmholtz equation using the dual variational method of Evequoz and Weth
\cite{EW}.
The main idea of the method is to rewrite \eqref{maineq} as $u = \mathfrak{R}(|u|^{p-2}u)$ and after the substitution $v = |u|^{p-2}u$, we are 
looking for a function $v \in L^{p^\prime}(\R^N)$ satisfying 
\begin{equation}
\label{eqdual}
\mathfrak{R}(v)= |v|^{p^\prime -2}v,\qquad \textrm{in}\ \R^N ,
\end{equation}
for $\frac{2(k+2)}{k} < p<\dfrac{2N}{(N-4)_+}$ and $\frac{1}{p} + \frac{1}{p'} = 1$. Equation \eqref{eqdual} admits a variational structure and solutions can be found as critical points of the functional $J \in C^1 (L^{p^\prime}(\R^N),\R)$ defined by
$$J(v)= \dfrac{1}{p^\prime} \int_{\R^N} |v|^{p^\prime} dx - \frac{1}{2} \int_{\R^N} \bar{v} \mathfrak{R} v dx \,,$$
 where $\bar{v}$ is the complex conjugate of $v$. Note that $J$ is real valued since $\hat{G}$ is real, and therefore by the Plancherel theorem 
 \begin{equation}
 \int_{\R^N} \bar{v} \mathfrak{R} v dx =  \int_{\R^N} \bar{\hat{v}} \hat{G} \hat{v} d\xi = \int_{\R^N}  \hat{G} |\hat{v}|^2 d\xi \,.
 \end{equation}

\begin{proof}[Proof of Theorem \ref{mainthm}]
Since $p > 2$, then $p' < 2$ and by a standard scaling argument (see e.g. \cite[Lemma 4.2]{EW}), one can show that $J$ has the mountain pass geometry, and therefore 
$$c=\inf_{\gamma \in P} \max_{t\in [0,1]} J(\gamma (t)) >0,$$
where $P= \{\gamma \in C([0,1],L^{p^\prime}(\R^N)):\ \gamma (0)=0\ \textrm{ and }  J(\gamma (1))<0 \}$
Moreover, there exists a Palais-Smale sequence $(v_n) \subset L^{p^\prime} (\R^N)$, that is, $(v_n)$ satisfies 
$\sup_n |J(v_n)| < \infty$ and $J'(v_n) \to 0$ in $L^{p'}(\R^N)$. Since  $p' < 2$ as in \cite[Lemma 4.2]{EW}
we have that $(v_n)$ is bounded in $L^{p^\prime} (\R^N)$. In addition, since $\hat{G}$ is real, the convolution is a self-adjoint operator, and therefore
\begin{equation}\label{lwb}
\left(\dfrac{1}{p^\prime} - \dfrac{1}{2}\right) \int_{\R^N} \bar{v}_n (G\ast v_n) dx = J(v_n) - \frac{1}{p'}J'(v_n)[v_n] \to c \qquad \textrm{as}\quad n \to \infty \,. 
\end{equation}

Next, we show a non-vanishing property. More precisely, we prove that there exist $R>0$, $\zeta >0$, and a sequence $(x_n)_n \subset \R^N$ such that, up to a subsequence
\begin{equation}
\label{contranon}
\int_{B_R (x_n)} |v_n|^{p^\prime} dx \geq \zeta\quad  \textrm{for all} \quad n.
\end{equation}
First, notice that it is sufficient to prove \eqref{contranon} for sequence $v_n$ belonging to $\mathcal{S}(\R^N)$ the class of Schwartz function. 
Otherwise, we replace $v_n$ by $\tilde{v}_n \in \mathcal{S}(\R^N)$ with $\|v_n - \tilde{v}_n\|_{L^{p'}} \leq \frac{1}{n}$. Arguing as in \cite[proof of Theorem 3.1]{EW}, we obtain that \eqref{lwb}
holds true with $v_n$ and $c$ replaced respectively by $\tilde{v}_n$ and $c/2$, and \eqref{contranon} holds for $v_n$ if we prove it for $\tilde{v}_n$.
We proceed by contradiction and assume that
\begin{equation}
\label{contra}
\lim_{n\rightarrow \infty} \left(\sup_{y\in \R^N} \int_{B_\rho (y)} |v_n|^{p^\prime} dx\right)=0\ \textrm{for all} \quad  \rho>0.
\end{equation}
The same decomposition as in the proof of Theorem \ref{thminvopbound} yields 
\begin{equation}
\label{contrae1}
\int_{\R^N} \bar{v}_n \mathfrak{R} v_n= \int_{\R^N} \bar{v}_n G_1 \ast v_n dx + \int_{\R^N} \bar{v}_n G_2 \ast v_ndx.
\end{equation}
Using the estimate \eqref{estG2} for $G_1$, we proceed exactly as in  \cite[Lemma $3.2$]{EW}, just replacing $N - 2$ by $N-4$ to show that 
\begin{equation}
\label{contrae2} \int_{\R^N} \bar{v}_n [G_1 \ast v_n]dx \rightarrow 0\ as\ n\rightarrow \infty.
\end{equation}
For fixed $R = 2^M > 0$ specified below, denote $M_R= \R^N \backslash B_R$  and decompose 
$$ \int_{\R^N} \bar{v}_n [G_2 \ast v_n] dx= \int_{\R^N} \bar{v}_n [(1_{B_R}G_2) \ast v_n] dx + \int_{\R^N} \bar{v}_n [(1_{M_R}G_2) \ast v_n] dx.$$
By the second half of \cite[Proof of Lemma 3.4]{EW} which only uses the boundedness of $G_2$, we obtain
\begin{equation}
\label{contrae3}
\lim_{n\rightarrow \infty}  \int_{\R^N} \bar{v}_n [(1_{B_R}G_2) \ast v_n] dx=0 \qquad \textrm{ for\ any} \quad \ R>0.
\end{equation}
To estimate $ \int_{\R^N} \bar{v}_n [(1_{M_R}G_2) \ast v_n] dx$,  denote $P_R= 1_{M_R} G_2$, $R\geq 4$ and for $\eta_j$ as in the proof of Theorem \ref{thminvopbound}, 
define
$$
 P_j(x)=P_R (x) \eta_j (x),\ j\in \N\,, \qquad \textrm{and therefore }
P_R=\sum_{j=[\log_2 R]}^\infty P^j.$$
Notice that $P^j = 0$ for $j \leq M - 1$ and $P^j = G^j_2$ for $j \geq M+1$, where $G^j_2 = G_2\eta_j$ was defined in the proof of Theorem \ref{thminvopbound}.
Since  $P^M$ can be treated as $1_{M_R}G_2$ above, we only need to estimate $P^j = G^j_2$ for $j \geq M + 1$. 
By \eqref{gutiee1}, 
\begin{equation}
\label{decPj}
\|P^j\|_{L^\infty (\R^N)} \leq C 2^{-jk/2},\qquad  \textrm{for any } \quad j\geq M + 1.
\end{equation}
Fix any $j > M$ and denote $d = \frac{2(k+2)}{k+4}$. Then, from \eqref{eq:resolvent_estimate} follows
$$\|P^j \ast v_n \|_{L^2 (\R^N)} \leq 2^{j/2} \|v_n\|_{L^{d} (\R^N)}, $$
and by duality (convolution with kernel that has real fourier transform is self-adjoint),
$$\|P^j \ast v_n \|_{L^{d'} (\R^N)} \leq 2^{j/2} \|v_n\|_{L^{2} (\R^N)}.$$
Since 
\begin{equation}\label{din}
\frac{1}{s} = \frac{1 - \theta}{b} + \frac{\theta}{c} \qquad \textrm{implies }  \frac{1}{s'} = \frac{1 - \theta}{c'} + \frac{\theta}{b'}
\end{equation}
if $\theta = \frac{1}{2}$, we obtain that for $(b, c) = (d, 2)$ that $s = \frac{2 (k+2)}{k+1}$. Consequently, the Riesz-Thorin theorem yields
$$\|P^j \ast v_n\|_{L^s (\R^N)} \leq C 2^{j/2} \|v_n\|_{L^{s^\prime}(\R^N)}.$$
On the other hand, from Young's inequality and \eqref{decPj} follows
$$\|P^j \ast v_n\|_{L^\infty (\R^N)} \leq C 2^{-jk/2}\|v_n\|_{L^1 (\R^N)}.$$
Since 
\begin{equation}
\frac{1}{p} = \frac{1 - \theta}{s} + \frac{\theta}{\infty} \qquad \textrm{implies } \frac{1}{p'} = \frac{1 - \theta}{s'} + \frac{\theta}{1}
\end{equation}
we obtain from the Riesz-Thorin theorem that, for any  $p\geq s$ and $j > M$,
$$\|P^j \ast v_n\|_{L^p (\R^N)} \leq C 2^{j (\frac{k+2}{p}  - \frac{k}{2}) } \|v_n\|_{L^{p^\prime}(\R^N)}.$$
Notice that, by assumption, $\frac{k+2}{p}  - \frac{k}{2}<0$, so a summation with respect to $j > M$ implies
\begin{align}
\|(1_{M_R} G_2) \ast v_n\|_{L^p (\R^N)} &\leq C \|v_n\|_{L^{p^\prime}(\R^N)} \sum_{j=M + 1}^\infty 2^{j (\frac{k+2}{p}  - \frac{k}{2}) }\\
 &\leq C2^{ M(\frac{k+2}{p}  - \frac{k}{2}) }   \|v_n\|_{L^{p^\prime}(\R^N)},
\end{align}
and consequently 
\begin{align}
\label{contrae4}
\sup_{n\in \N} \left|\int_{\R^N} \bar{v}_n [(1_{M_R} G_2) \ast v_n] dx\right| &\leq C2^{M(\frac{(k+2)}{p}  - \frac{k}{2}) }\sup_{n\in \N} \|v_n\|_{L^{p^\prime}(\R^N)}^2
\rightarrow 0
\end{align}
 as $R\rightarrow \infty$.
Thus, we have proved that if \eqref{contra} holds then, using \eqref{contrae1}, \eqref{contrae2}, \eqref{contrae3}, and \eqref{contrae4},
$$\int_{\R^N} v_n \mathfrak{R} v_n dx \rightarrow 0,\ as\ n\rightarrow 0. $$
This contradicts \eqref{lwb}, and therefore \eqref{contranon} holds. Thus, denoting $u_n (x)=v_n (x-x_n) $, $(u_n)$ is a bounded Palais-Smale sequence of $J$ which weakly converges to some 
$u\in L^{p^\prime} (\R^N)$. By proceeding as for instance in \cite[Theorem 4.1]{MR3977892}, for any $R > 0$ any any smooth $\varphi$ compactly supported in $B_R$ one has
\begin{multline}
\left| \int_{\R^N} (|u_n|^{p'-2}u_n - |u_m|^{p'-2}u_m)\varphi dx\right| \\
\begin{aligned}
&= \left|J'(u_n)[\varphi] - J'(m_n)[\varphi] + \int_{B_R} \mathfrak{R} (u_n - u_m)\varphi \right|\\
&\leq (\|J'(u_n)\| + \|J'(u_m)\|) \|\varphi\|_{L^{p'}(\R^N)} + \|1_{B_R} \mathfrak{R} (u_n - u_m)\|_{L^p(\R^N)} \|\varphi\|_{L^{p'}(\R^N)} \,.
\end{aligned}
\end{multline}
Since $p < \frac{2N}{(N-4)_+}$, $W^{4, p'}$  is compactly embedded in $L^p$, and by local regularity results (see \cite[Theorem $14.1'$]{ADN_estimates_I}) one has 
\begin{equation}
 \|1_{B_R} \mathfrak{R} w\|_{L^{p}(\R^N)} \leq C  \|1_{B_R} \mathfrak{R} w\|_{W^{4, p'}(\R^N)}  \leq C_R \|w \|_{L^{p'}(B_R)} .
\end{equation}
Therefore, by compactness $(\mathfrak{R} (u_n - u_m))_{m, n}$ converges strongly to zero as $m, n \to \infty$ in $L^{p}$. 
In addition, both $\|J'(u_n)\|$ and $\|J'(u_m)\|$ converge to zero as $n, m \to 0$. Thus $|u_n|^{p' - 2}u_n$ strongly converges to $|u|^{p' - 2}u$ locally in $L^p$. 
By \eqref{contranon}, we have
\begin{equation}
\zeta \leq \int_{B_R(x_n)} |v_n|^{p'} dx = \int_{B_R} |u_n|^{p'} dx =  \int_{B_R} |u_n|^{p'-2}u_n \overline{u}_n dx \to  \int_{B_R} |u|^{p'} dx
\end{equation}
as $n \to \infty$, were we used that $|u_n|^{p' - 2}u_n \to |u|^{p' - 2}u$ strongly in $L^p(B_R)$ and $u_n \rightharpoonup u$ weakly in $L^{p'}(\R^N)$. By standard arguments
(see for example \cite[Theorem 4.1]{MR3977892}), we obtain that $u\in L^{p'}(\R^N)$ is a non-trivial critical point of $J$. Also, by 
\cite[Theorem $14.1'$]{ADN_estimates_I}, see for example  \cite[Proposition 5.1]{MR3977892}), we obtain that $u \in W^{4, p}_{\textrm{loc}}(\R^N)$, and therefore by  \eqref{contranon}, $u$ is a nontrivial (strong) solution of \eqref{maineq}. 
 
 To obtain 
 the global estimates, we proceed as in  \cite[Proposition 5.2 and Theorem 5.1]{MR3977892}. By a bootstrap argument and using again that $p < \frac{2N}{(N-4)_+}$, we 
 obtain that $u$ is bounded (by a function of $\|u\|_{L^{p'}(\R^N)} < \infty$, which is bounded), for details see \cite[proof of Theorem 5.1]{MR3977892}, and consequently by interpolation, 
 $u \in L^q(\R^N)$ for any $q \in [p, \infty]$.  Finally, 
by using \cite[Corollary on page 559]{MR1955094}, we obtain that $u \in W^{4, q}(\R^N)$ for any $q \in [p, \infty)$ and the H\" older regularity follows from 
embeddings of Sobolev into H\" older spaces. 

This concludes the proof.
\end{proof}


\begin{thebibliography}{10}

\bibitem{ADN_estimates_I}
S.~Agmon, A.~Douglis, and L.~Nirenberg.
\newblock Estimates near the boundary for solutions of elliptic partial
  differential equations satisfying general boundary conditions. {I}.
\newblock {\em Comm. Pure Appl. Math.}, 12:623--727, 1959.

\bibitem{MR2831841}
J.-G. Bak and A.~Seeger.
\newblock Extensions of the {S}tein-{T}omas theorem.
\newblock {\em Math. Res. Lett.}, 18(4):767--781, 2011.

\bibitem{MR1745182}
M.~Ben-Artzi, H.~Koch, and J.-C Saut.
\newblock Dispersion estimates for fourth order {S}chr\"odinger equations.
\newblock {\em C. R. Acad. Sci. Paris S\'er. I Math.}, 330(2):87--92, 2000.

\bibitem{MR3855391}
D.~Bonheure, J.-B. Casteras, E.~M. dos Santos, and R.~Nascimento.
\newblock Orbitally stable standing waves of a mixed dispersion nonlinear
  {S}chr\"{o}dinger equation.
\newblock {\em SIAM J. Math. Anal.}, 50(5):5027--5071, 2018.

\bibitem{MR3976588}
D.~Bonheure, J.-B. Casteras, T.~Gou, and L.~Jeanjean.
\newblock Normalized solutions to the mixed dispersion nonlinear
  {S}chr\"{o}dinger equation in the mass critical and supercritical regime.
\newblock {\em Trans. Amer. Math. Soc.}, 372(3):2167--2212, 2019.

\bibitem{MR4001029}
D.~Bonheure, J.-B. Casteras, T.~Gou, and L.~Jeanjean.
\newblock Strong instability of ground states to a fourth order
  {S}chr\"{o}dinger equation.
\newblock {\em Int. Math. Res. Not. IMRN}, (17):5299--5315, 2019.

\bibitem{MR3977892}
D.~Bonheure, J.-B. Casteras, and R.~Mandel.
\newblock On a fourth-order nonlinear {H}elmholtz equation.
\newblock {\em J. Lond. Math. Soc. (2)}, 99(3):831--852, 2019.

\bibitem{MR3494890}
D.~Bonheure and R.~Nascimento.
\newblock Waveguide solutions for a nonlinear {S}chr\"{o}dinger equation with
  mixed dispersion.
\newblock In {\em Contributions to nonlinear elliptic equations and systems},
  volume~86 of {\em Progr. Nonlinear Differential Equations Appl.}, pages
  31--53. Birkh\"{a}user/Springer, Cham, 2015.

\bibitem{BL}
T.~Boulenger and E.~Lenzmann.
\newblock Blowup for biharmonic {NLS}.
\newblock {\em Ann. Sci. \'Ec. Norm. Sup\'er. (4)}, 50(3):503--544, 2017.

\bibitem{BLSS}
L.~Bugiera, E.~Lenzmann, A.~Schikorra, and J.~Sok.
\newblock On symmetry of traveling solitary waves for dispersion generalized
  {NLS}.
\newblock {\em Nonlinearity}, 33(6):2797--2819, 2020.

\bibitem{C}
J.-B. Casteras.
\newblock Travelling wave solutions for a fourth order schr\" odinger equation
  with mixed dispersion.
\newblock {\em Preprint available on https
  ://sites.google.com/view/jeanbaptistecasteras/accueil}, 2020.

\bibitem{Ev}
G.~Ev\'equoz.
\newblock Existence and asymptotic behavior of standing waves of the nonlinear
  {H}elmholtz equation in the plane.
\newblock {\em Analysis (Berlin)}, 37(2):55--68, 2017.

\bibitem{EW}
G.~Evequoz and T.~Weth.
\newblock Dual variational methods and nonvanishing for the nonlinear
  {H}elmholtz equation.
\newblock {\em Adv. Math.}, 280:690--728, 2015.

\bibitem{FJMM}
A.~J. Fern\'{a}ndez, L.~Jeanjean, R.~Mandel, and M.~Maris.
\newblock Some non-homogeneous gagliardo-nirenberg inequalities and application
  to a biharmonic non-linear schr\" odinger equation.
\newblock {\em Preprint arXiv:2010.01448}, 2020.

\bibitem{MR1898529}
G.~Fibich, B.~Ilan, and G.~Papanicolaou.
\newblock Self-focusing with fourth-order dispersion.
\newblock {\em SIAM J. Appl. Math.}, 62(4):1437--1462 (electronic), 2002.

\bibitem{Grafakos}
L.~Grafakos.
\newblock {\em Classical {F}ourier analysis}, volume 249 of {\em Graduate Texts
  in Mathematics}.
\newblock Springer, New York, third edition, 2014.

\bibitem{Gu}
S.~Guti\'errez.
\newblock Non trivial {$L^q$} solutions to the {G}inzburg-{L}andau equation.
\newblock {\em Math. Ann.}, 328(1-2):1--25, 2004.

\bibitem{MR3556359}
K.~Hambrook and I.~{\L}aba.
\newblock Sharpness of the {M}ockenhaupt-{M}itsis-{B}ak-{S}eeger restriction
  theorem in higher dimensions.
\newblock {\em Bull. Lond. Math. Soc.}, 48(5):757--770, 2016.

\bibitem{Herz}
C.~S. Herz.
\newblock Fourier transforms related to convex sets.
\newblock {\em Ann. of Math. (2)}, 75:81--92, 1962.

\bibitem{Him}
D.~Himmelsbach.
\newblock Blowup, solitary waves and scattering for the fractional nonlinear
  schr\" odinger equation.
\newblock {\em PhD Thesis University of Basel}, 2017.

\bibitem{MR1779828}
V.I. Karpman and A.G. Shagalov.
\newblock Stability of solitons described by nonlinear {S}chr\"odinger-type
  equations with higher-order dispersion.
\newblock {\em Phys. D}, 144(1-2):194--210, 2000.

\bibitem{MR155146}
W.~Littman.
\newblock Fourier transforms of surface-carried measures and differentiability
  of surface averages.
\newblock {\em Bull. Amer. Math. Soc.}, 69:766--770, 1963.

\bibitem{MR3949725}
R.~Mandel.
\newblock The limiting absorption principle for periodic differential operators
  and applications to nonlinear {H}elmholtz equations.
\newblock {\em Comm. Math. Phys.}, 368(2):799--842, 2019.

\bibitem{MMP}
R.~Mandel, E.~Montefusco, and B.~Pellacci.
\newblock Oscillating solutions for nonlinear {H}elmholtz equations.
\newblock {\em Z. Angew. Math. Phys.}, 68(6):68:121, 2017.

\bibitem{miao}
C.~Miao, G.~Xu, and L.~Zhao.
\newblock Global well-posedness and scattering for the defocusing
  energy-critical nonlinear {S}chr\"odinger equations of fourth order in
  dimensions {$d\geqslant9$}.
\newblock {\em J. Differential Equations}, 251(12):3381--3402, 2011.

\bibitem{MR1882456}
T.~Mitsis.
\newblock A {S}tein-{T}omas restriction theorem for general measures.
\newblock {\em Publ. Math. Debrecen}, 60(1-2):89--99, 2002.

\bibitem{MR1955094}
Y.~Miyazaki.
\newblock The {$L^p$} resolvents of elliptic operators with uniformly
  continuous coefficients.
\newblock {\em J. Differential Equations}, 188(2):555--568, 2003.

\bibitem{MR1810754}
G.~Mockenhaupt.
\newblock Salem sets and restriction properties of {F}ourier transforms.
\newblock {\em Geom. Funct. Anal.}, 10(6):1579--1587, 2000.

\bibitem{MR2353631}
B.~Pausader.
\newblock Global well-posedness for energy critical fourth-order
  {S}chr\"odinger equations in the radial case.
\newblock {\em Dyn. Partial Differ. Equ.}, 4(3):197--225, 2007.

\bibitem{Tanabe1997}
H.~Tanabe.
\newblock {\em Functional analytic methods for partial differential equations},
  volume 204 of {\em Monographs and Textbooks in Pure and Applied Mathematics}.
\newblock Marcel Dekker, Inc., New York, 1997.

\end{thebibliography}
\end{document}